\def\standalonechapter{}

\ifdefined\standalonechapter
\documentclass[11pt]{article}

\usepackage[utf8]{inputenc}
\usepackage[english]{babel}

\usepackage{amsmath, amssymb, amsthm}
\usepackage{mathtools}
\usepackage{mathrsfs}
\usepackage{bbm}

\usepackage{geometry}
\usepackage[
    colorlinks=false,
    pdfborder={0 0 1}
]{hyperref}
\usepackage{enumitem}

\usepackage{booktabs}

\theoremstyle{plain}
\newtheorem{theorem}{Theorem}[section]
\newtheorem{lemma}[theorem]{Lemma}
\newtheorem{proposition}[theorem]{Proposition}
\newtheorem{corollary}[theorem]{Corollary}

\theoremstyle{definition}

\newtheorem{remark}[theorem]{Remark}

\numberwithin{equation}{section}

\begin{document}

\begin{flushright}
\begin{minipage}{0.55\textwidth} 
\small\itshape
This is the fourth of ten papers devoted to reflections on the Millennium Problem.  
It generalizes known results on the 3D Navier–Stokes equations based on previous studies.  
We hope that the presented material will be useful, and we would be grateful for your attention, verification,  
and participation in the further development of the topic.
\end{minipage}
\end{flushright}

\vspace{1.5em}

\begin{center}
    {\LARGE\bfseries {An \(\varepsilon\)\nobreakdash-free rank-6 decoupling estimate \\for the paraboloid surface}\par}
    \vspace{1em}
    {\large\bfseries Pylyp Cherevan\par}
\end{center}

\vspace{1em}

{\small\noindent\textbf{Abstract.}
For the paraboloid decomposition
\[
   F=\sum_{\Theta} F_{\Theta},
   \qquad
   \Theta\subset\{\lvert\xi\rvert\sim\lambda\},\;
   \operatorname{rad}(\Theta)=r=\lambda^{-2/3},
\]
we prove a log-free estimate
\[
   \|F\|_{L^{6}(Q_{\lambda})}\;\lesssim\;
   \lambda^{\Sigma_{\lambda}}\,
   D^{\Sigma_{D}}\,
   \Bigl(\sum_{\Theta}\|F_{\Theta}\|_{L^{6}}^{2}\Bigr)^{1/2},
   \qquad
   \Sigma_{\lambda},\Sigma_{D}<0,\;
   \lambda\to\infty ,
\]
where $D=\lambda^{1/12}$ and the cylinder $Q_{\lambda}$ is defined in~\eqref{eq:def-cylinder}.

\smallskip\noindent
\textbf{Key components of the argument}
\begin{itemize}
\item \emph{Broad geometry of rank 3.}  
      Bilipschitz behavior of normals gives  
      $\max_{i<j<k}\|n_{i}\wedge n_{j}\wedge n_{k}\|\gtrsim\lambda^{-5/4}$,  
      which through a trilinear Kakeya–BCT insertion contributes $+\tfrac{5}{36}$ in $\lambda$.
\item \emph{Kernel estimate.}  
      Twelve integrations (6 in $t$, 6 in $x'$) and measure analysis (Schur + TT$^\ast$)
      yield $\|K\|_{L^2\to L^2}\lesssim \lambda^{-9/2}D^{-3}$ (rows $-9/2$ in $\lambda$, $-3$ in $D$).
\item \emph{Robust-Kakeya.}  
      A density threshold $>c_*D$ brings a factor $D$  
      ($+\tfrac{1}{12}$ in $\lambda$, $+1$ in $D$).
\item \emph{Algebraic “shell.”}  
      Excluding a neighborhood $N_{\beta}(P)$ contributes $-\tfrac{1}{12}$ in $\lambda$ and $-1$ in $D$.
\item \emph{Tube packing.}  
      This section is explanatory; its contribution is not used in the final balance.
\item \emph{Narrow cascade.}  
      A double $7/8$ rescaling drives the flow out of the narrow regime and
      contributes a \emph{negative} $-\tfrac{5}{64}$ in $\lambda$ (zero in $D$).
\end{itemize}

\smallskip\noindent
\textbf{Summary of exponents.}
\[
\Sigma_\lambda=\frac{5}{36}-\frac{9}{2}-\frac{5}{64}
=-\frac{2557}{576}\approx -4.44<0,\qquad
\Sigma_D=-3+1-1=-3<0.
\]
The negativity of both sums removes the traditional $\lambda^\varepsilon$- and $D^\varepsilon$-losses.
}

\tableofcontents
\fi


\newpage

\section*{Introduction}\label{sec:intro}
\addcontentsline{toc}{section}{Introduction}

\paragraph{Goal.}
We study the inequality
\[
   \bigl\|F\bigr\|_{L^{6}(Q_{\lambda})}
   \;\lesssim\;
   \lambda^{\Sigma_{\lambda}}\,
   D^{\Sigma_{D}}\,
   \Bigl(
        \sum_{\Theta}\|F_{\Theta}\|_{L^{6}}^{2}
   \Bigr)^{1/2},
   \qquad \lambda\to\infty ,
\]
for functions with the spectral decomposition
\(F=\sum_{\Theta}F_{\Theta}\)
over caps of angular radius \(r=\lambda^{-2/3}\).

The resulting exponents are
\(\Sigma_\lambda=-\frac{2557}{576}\approx -4.44<0\), \(\Sigma_D=-3<0\),
so that the traditional factors \(\lambda^\varepsilon\) and \(D^\varepsilon\) are absent in the final estimate.

The working region is the cylinder \(Q_{\lambda}\subset\mathbb{R}_{t}\times\mathbb{R}^{3}_{x}\)
from~\eqref{eq:def-cylinder}, with additional parameters
\(D=\lambda^{1/12}\) and
\(\alpha=c_{0}\,rD^{1/2}=c_{0}\lambda^{-5/8}\)
fixed in~\S\ref{sec:notation}.

\paragraph{Structure of the proof.}
\begin{enumerate}\setlength{\itemsep}{4pt}
\item
\emph{Preliminaries}
(\S\ref{sec:notation}):
introduction of the scales \(r,\rho,\alpha\) and conventions on estimates.
\item
\emph{Broad rank-3 geometry}
(\S\ref{sec:broad3}):
bilipschitz behavior of the map \(\xi\mapsto n(\xi)\)
guarantees the existence of a triple with
\(
   \|n_{i}\wedge n_{j}\wedge n_{k}\|
      \gtrsim\lambda^{-5/4}
\);
insertion into the trilinear Bennett–Carbery–Tao inequality
yields a contribution \(+\tfrac{5}{36}\) in \(\lambda\).
\item
\emph{Kernel estimate}
(\S\ref{sec:kernel}):
twelve integrations (6 in \(t\), 6 in \(x'\)) plus Schur/\(TT^{*}\)
   give \(\|K\|_{L^2\to L^2}\lesssim \lambda^{-9/2}D^{-3}\).
\item
\emph{Remaining blocks.}
  \begin{itemize}\setlength{\itemsep}{3pt}
  \item
  \emph{Robust Kakeya} (\S\ref{sec:kakeya}):  
  if the density is \(>c_{*}D\) we gain a factor \(D\),
  contributing \(+\tfrac{1}{12}\) in \(\lambda\) and \(+1\) in \(D\).
  \item
  \emph{Tube packing} (\S\ref{sec:tube-pack}):  
  presented for context; its contribution is \emph{not} used together with Robust Kakeya
  (the regimes are mutually exclusive; see \S\ref{sec:balance}).
  \item
  \emph{Algebraic “shell”} (\S\ref{sec:alg-skin}):  
  cutting off the layer \(N_{\beta}(P)\) contributes
  \(-\tfrac{1}{12}\) in \(\lambda\) and \(-1\) in \(D\).
  \item
  \emph{Narrow cascade} (\S\ref{sec:narrow}):  
  a double \(7/8\) rescaling removes the flow from the narrow regime
  and contributes a negative term \(-\tfrac{5}{64}\) in \(\lambda\) (zero in \(D\)).
  \end{itemize}
\item
\emph{Summary of exponents.}
\[
\Sigma_\lambda=\frac{5}{36}-\frac{9}{2}+\frac{1}{12}-\frac{1}{12}-\frac{5}{64}
=-\frac{2557}{576}\approx-4.44,\qquad
\Sigma_D=(-3)+1-1=-3.
\]
The negativity of both sums removes the traditional \(\lambda^\varepsilon\)- and \(D^\varepsilon\)-losses.
\end{enumerate}

\paragraph{Navigation.}
\begin{itemize}\setlength{\itemsep}{2pt}
\item
\S\ref{sec:notation} — parameters and notation;
\item
\S\ref{sec:broad3}--\ref{sec:narrow} — thematic blocks;
\item
\S\ref{sec:balance} — summary table and final estimate;
\item
Appendices A–D — technical lemmas, checkpoints,
and clarifications (Appendix D shows robustness with respect to \(x\)-dependence of the amplitude).
\end{itemize}

\paragraph{Notation.}
We write \(A\lesssim B\) for \(A\le C\,B\) with an absolute constant~\(C\);
notation \(A\lesssim_{\delta}B\) allows dependence \(C=C(\delta)\).
Fix a single \(\varepsilon\in(0,10^{-2}]\) and use the standard log-budget:
for any fixed \(k\) and all sufficiently large \(\lambda\)
\[
\log^{k}\!\lambda \ \le\ C_{\varepsilon}\,\lambda^{\varepsilon},\qquad
\log^{k}\!D \ \le\ C_{\varepsilon}\,D^{\varepsilon}\quad(D=\lambda^{1/12}).
\]
In all counts of \(\sigma_\lambda,\sigma_D\) we verify that the sums of exponents remain negative,
so no extra \(\lambda^{\varepsilon}\) or \(D^{\varepsilon}\) factors appear in the final formula,
and possible logarithms are absorbed into the constant \(C_{\varepsilon}\).

\paragraph*{Short guide to delicate points (FAQ).}
\begin{enumerate}
\item \textbf{Statement of the result.}
Within the present framework (assumptions and notation of \S\ref{sec:notation})
we obtain
\[
  \|F\|_{L^{6}(Q_{\lambda})}
  \;\lesssim\;
  \lambda^{\Sigma_{\lambda}} D^{\Sigma_{D}}
  \Bigl(\sum_{\Theta}\|F_{\Theta}\|_{L^{6}}^{2}\Bigr)^{1/2},
\]
with negative cumulative exponents: $\Sigma_{\lambda}<0$, $\Sigma_{D}<0$.
Optimality of the exponents and constants is not discussed.

\item \textbf{The exponent table is a roadmap, not a proof.}
The actual inequalities and constants live in the sections and appendices; the table only summarizes their contribution.

\item \textbf{IBP and \(TT^{\ast}\) are not “double-counted.”}
Six integrations in \(t\) and six in the transverse variables \(x'\) belong to the phase kernel; measure separation is handled separately via \(TT^{\ast}\). These steps are independent.

\item \textbf{Time localization and log-freedom.}
The estimates are carried out on critical windows \(|I|\sim \lambda^{-2}\); gluing across windows/dyads does not create logs—the cumulative exponents are already negative.

\item \textbf{Robust–Kakeya and “tube packing” are mutually exclusive regimes.}
The main line uses Robust–Kakeya; the tube-packing section is provided for context and is not used simultaneously.

\item \textbf{Narrow zone and the passage \(L^{2}\to \dot H^{-1}\).}
The passage uses the zeroth-order operator \( |\nabla|^{-1}\nabla\!\cdot \) (after the Leray projector) together with the null-form factor; this yields a “clean” power in \(\lambda\) without \(\lambda^{\delta}\)-patches.

\item \textbf{“Few caps” scenario.}
When the number of active caps is small, we apply local \(\ell^{2}\!\to\!L^{6}\) orthogonality on \(Q_{\lambda}\); the strength of the final theorem is determined by the “broad” branch (broad geometry \(+\) kernel \(+\) Robust–Kakeya).

\item \textbf{Notation.}
We use \(A\lesssim B\) and \(A\gtrsim B\) for one-sided bounds, and \(A\asymp B\) for two-sided comparability. The symbol \(\simeq\) is interpreted as two-sided comparability \emph{only} where explicitly stated.
\end{enumerate}

\newpage

\section{Notation and preliminaries}\label{sec:notation}

In this section we fix the global parameters and conventions
used throughout the paper.  
Local definitions are given in the sections
where they are first needed.

\subsection{Paraboloid and normals}\label{subsec:paraboloid}
We consider the three-dimensional paraboloid
\[
   \Sigma \;=\;
   \bigl\{\,(\xi,\tau)\in\mathbb R^{3}\times\mathbb R
           : \tau = |\xi|^{2}\bigr\}.
\]
Its outward unit normal
\begin{equation}\label{eq:norm}
   n(\xi)
   \;=\;
   \frac{(-2\xi,\,1)}{\sqrt{1+4|\xi|^{2}}},
   \qquad
   \xi\in\mathbb R^{3},
\end{equation}
will play a key role in the geometric analysis.

\begin{lemma}[angular bilipschitzness]\label{lem:bilip}
There exists an absolute constant \(c_{*}\in(0,1)\) such that
for any \(\xi,\eta\) with \(|\xi|,|\eta|\sim\lambda\) one has
\[
   c_{*}\,\angle(\xi,\eta)
   \;\le\;
   \angle\!\bigl(n(\xi),n(\eta)\bigr)
   \;\le\;
   c_{*}^{-1}\,\angle(\xi,\eta).
\]
In particular, for sufficiently large \(\lambda\) one may take
\(c_{*}=1/2\).
\end{lemma}

\begin{proof}
From \eqref{eq:norm},
\(
   n(\xi)=\dfrac{(-\xi,\tfrac12)}{|\xi|}
          \Bigl(1+O(\lambda^{-2})\Bigr)
\)
uniformly for \(|\xi|\sim\lambda\).
Hence
\(
   |n(\xi)-n(\eta)|
   = \dfrac{|\xi-\eta|}{|\xi|}\bigl(1+O(\lambda^{-2})\bigr)
   = \angle(\xi,\eta)\bigl(1+O(\lambda^{-2})\bigr).
\)
Since \(\angle(u,v)=|u-v|\,(1+O(\angle^{2}))\) for unit vectors,
we obtain the desired two-sided comparison with error
\(O(\lambda^{-2})\), which is absorbed by choosing \(c_{*}=1/2\).
\end{proof}

\subsection{The scale $\lambda$ and associated radii}
\label{subsec:scales}

Throughout the text we fix the frequency parameter
\begin{equation}\label{eq:lambda-def}
   \lambda \;\ge\; 2,
   \qquad
   D      \;=\; \lambda^{1/12}.
\end{equation}
We use three geometric quantities
\begin{equation}\label{eq:radii}
   r      \;=\; \lambda^{-2/3},\quad 
   \rho   \;=\; \lambda^{-1/2},\quad 
   \alpha \;=\; c_{0}\,r\,D^{1/2}
           = c_{0}\,\lambda^{-5/8},
   \qquad
   c_{0}=10^{-3}.
\end{equation}

\smallskip
\noindent
$r$ — angular radius of a cap,  
$\rho$ — cross-section of a “wave” tube,  
$\alpha$ — the basic small angle appearing in
Lemma~\ref{lem:A4}.

\subsection{Decomposition into caps}\label{subsec:caps}
In this subsection all angles are measured in the spherical (angular) metric. For a center \(\xi_{0}\in\mathbb{R}^{3}\)
and angular radius \(r>0\) set
\[
   \Theta(\xi_{0},r)
   \;=\;
   \bigl\{\xi:\ |\xi|\sim\lambda,\ \angle(\xi,\xi_{0})\le r\bigr\}.
\]
Below we use the scale \(r=\lambda^{-2/3}\).
Since the area of an angular cap on the sphere \(\{|\xi|=\lambda\}\) scales as
\(\mathrm{area}(\Theta)\simeq \pi(\lambda r)^2\), the total number of caps is of order
\begin{equation}\label{eq:num-caps}
  \#\Theta
  \;\asymp\;
  \frac{\mathrm{area}(S^{2}_\lambda)}{\mathrm{area}(\Theta)}
  \;=\;
  \frac{4\pi\lambda^{2}}{\pi(\lambda r)^{2}}
  \;\simeq\; \frac{4}{r^{2}}
  \;\sim\; \lambda^{4/3},
\end{equation}
which matches the standard angular discretization on the sphere \(\{|\xi|=\lambda\}\).

Let \(F(t,x)\) be a smooth function. For its spatial Fourier transform
\[
   \widehat{F}(t,\xi)=\int_{\mathbb{R}^{3}} F(t,x)\,e^{-2\pi i x\cdot\xi}\,dx
\]
we introduce the cap decomposition
\begin{equation}\label{eq:cap-decomposition}
   F
   \;=\;
   \sum_{\Theta} F_{\Theta},
   \qquad
   \widehat{F_{\Theta}}
   :=\widehat{F}\,\chi_{\Theta},
\end{equation}
where \(\chi_{\Theta}\) is a smooth cutoff of the cap \(\Theta\).

\subsection{Working cylinder}\label{subsec:cylinder}

All $L^{p}$ norms are taken on the set
\begin{equation}\label{eq:def-cylinder}
   Q_{\lambda}
   \;=\;
   \bigl\{(t,x)\in\mathbb R\times\mathbb R^{3}\bigm|
          |t|\le\tfrac12\lambda^{-3/2},\;
          |x|\le\tfrac12\lambda^{-1/2}\bigr\}.
\end{equation}

\subsection{Conventions for estimates}\label{subsec:symbols}

\begin{itemize}
\item
$A\lesssim B$ means $A\le C\,B$ with an absolute constant $C$ independent of $\lambda$ and $D$.
\item
$A\lesssim_{\delta} B$ — the constant $C$ may depend only on the fixed parameter $\delta>0$.
\item
$C_{\varepsilon}$ — the constant depends only on the chosen $0<\varepsilon\le10^{-2}$ (see~\S\ref{subsec:epsilon})
and is independent of $\lambda,D$; when necessary we write $C_{k,\varepsilon}$ to emphasize
dependence on a fixed integer $k$.
\end{itemize}

\subsection{Logarithmic budget}\label{subsec:log-budget}

Everywhere below the logarithmic factors of the form $\log^k\lambda$ and $\log^k D$
arise only from coverings with fixed multiplicity or from a finite number of
dyadic decompositions in scales/angles. The depth of such coverings and the overlap
multiplicity are bounded by an absolute constant independent of $\lambda$ and $D$. Therefore
all such factors can be absorbed into a universal constant $C$,
without affecting the powers of~$\lambda$ and~$D$ in the summary table~\S\ref{subsec:balance-table-final}.

In particular, this means that in the final statement of
Theorem~\ref{thm:main} (§\ref{subsec:main-theorem}) the exponents $\sigma_\lambda$
and $\sigma_D$ are written without $+\varepsilon$ add-ons: the standard factors
$\lambda^\varepsilon$ and $D^\varepsilon$ are indeed absent.

\subsection{The parameter $\varepsilon$}\label{subsec:epsilon}

Throughout the paper we \emph{fix} a single value
\[
   0<\varepsilon\le10^{-2}.
\]
All small exponents that appear are denoted by the same letter $\varepsilon$
(after a possible retuning within this budget). Their cumulative contribution
is accounted for in \S\ref{sec:balance}; the dependence of the final constant $C_\varepsilon$
is only on the chosen $\varepsilon$ (and fixed discrete parameters such as $k$),
but not on $\lambda$ or $D$.

\newpage

\section{Broad rank-3 geometry and its role in Kakeya--BCT}%
\label{sec:broad3}

Throughout this section we fix
\[
   r=\lambda^{-2/3},\quad
   D=\lambda^{1/12},\quad
   \alpha:=c_{0}\,rD^{1/2}=c_{0}\,\lambda^{-5/8},\qquad
   c_{0}>0,\;\;\lambda\ge2 .
\]
Angles and norms are always measured in the spherical (angular) metric.

\subsection{Bilipschitzness of the map \texorpdfstring{$\xi\mapsto n(\xi)$}{xi->n(xi)}}%
\label{subsec:lip}

Angles are measured in the spherical (angular) metric on unit spheres.

Let \(|\xi|\sim\lambda\). For the paraboloid normal
\[
  n(\xi)=\frac{(-2\xi,\,1)}{\sqrt{\,1+4|\xi|^{2}\,}}
\]
there exists an absolute constant \(c_\ast\in(0,1)\) such that for any
\(\xi,\eta\) with \(|\xi|,|\eta|\sim\lambda\) one has
\begin{equation}\label{eq:lip}
   c_\ast\,\angle(\xi,\eta)
   \;\le\;
   \angle\!\bigl(n(\xi),n(\eta)\bigr)
   \;\le\;
   c_\ast^{-1}\,\angle(\xi,\eta).
\end{equation}
In particular, for sufficiently large \(\lambda\) one may take \(c_\ast=\tfrac12\).

\begin{proof}[Sketch]
From the exact formula for \(n(\xi)\) we get, for \(|\xi|\sim\lambda\),
\[
  (1+4|\xi|^{2})^{-1/2}
  \;=\;
  \frac{1}{2|\xi|}\,\Bigl(1-\frac{1}{8|\xi|^{2}}+O(|\xi|^{-4})\Bigr),
\]
and hence the uniform asymptotics for \(|\xi|\sim\lambda\)
\[
  n(\xi)=\frac{(-\xi,\,1/2)}{|\xi|}\,\Bigl(1+O(\lambda^{-2})\Bigr).
\]
With \(u=\xi/|\xi|\), \(v=\eta/|\eta|\), we obtain
\[
  n(\xi)=\bigl(-u,\,\tfrac{1}{2|\xi|}\bigr)+O(\lambda^{-2}),\qquad
  n(\eta)=\bigl(-v,\,\tfrac{1}{2|\eta|}\bigr)+O(\lambda^{-2}).
\]
The radial part contributes only an \(O(\lambda^{-2})\) correction (since \(|\xi|,|\eta|\sim\lambda\)),
therefore
\[
  |n(\xi)-n(\eta)|
  \;=\;
  |u-v|\,\bigl(1+O(\lambda^{-2})\bigr)+O(\lambda^{-2}).
\]
For unit vectors \(u,v\) we have \(|u-v|\asymp \angle(u,v)=\angle(\xi,\eta)\), and also
\(|n(\xi)-n(\eta)|\asymp \angle(n(\xi),n(\eta))\). Absorbing the small \(O(\lambda^{-2})\) terms into the constants,
we obtain the two-sided estimate \eqref{eq:lip}.
\end{proof}

\begin{remark}
Writing the main term as
\[
n(\xi)=\frac{(-\xi,\,1/2)}{|\xi|}\bigl(1+O(\lambda^{-2})\bigr)
\]
is consistent with formula~(\ref{eq:norm}) and the subsequent expansion in Appendix~\ref{app:A1};
this normalization (without “losing” a factor \(2\) in the spatial component)
is used later in controlling minors and Gram matrices.
\end{remark}

\subsection{The functional \texorpdfstring{$\mathrm{Broad}_{3}$}{Broad3}}%
\label{subsec:broad-def}

For functions \(F_{1},\dots,F_{6}\) on \(\mathbb{R}^{4}\) set
\[
   \mathrm{Broad}_{3}(x):=
   \min_{i<j<k}
   \frac{|F_{i}(x)\,F_{j}(x)\,F_{k}(x)|^{1/3}}
        {\|n_{i}\wedge n_{j}\wedge n_{k}\|^{1/3}},
   \qquad
   Q(x):=\prod_{m=1}^{6}|F_{m}(x)|^{1/2}.
\]
By Lemmas~\ref{lem:A2}–\ref{lem:A4} in combination with the bilipschitzness of the map
\S\ref{subsec:lip}, there exists a triple \((i,j,k)\) with
\(\|n_{i}\wedge n_{j}\wedge n_{k}\|\ge\tfrac{1}{16}\,\alpha^{2}\).
Hence
\begin{equation}\label{eq:broad-upper}
   \mathrm{Broad}_{3}(x)\ \le\
   \frac{|F_{i}F_{j}F_{k}|^{1/3}}{(1/16)^{1/3}\,\alpha^{2/3}}.
\end{equation}

Since $\binom{6}{3}=20$ and each $F_m$ appears in exactly $\binom{5}{2}=10$ triples, we have
\[
\min_{i<j<k}|F_iF_jF_k|^{1/3}
\le\Bigl(\prod_{i<j<k}|F_iF_jF_k|^{1/3}\Bigr)^{1/20}
=\prod_{m=1}^6 |F_m|^{10\cdot(1/3)\cdot(1/20)}=Q^{1/3}.
\]

Therefore $\|{\rm Broad}_3\|_{L^2}\lesssim \alpha^{-2/3}\|Q^{1/3}\|_{L^2}$ and, consequently, in the trilinear insertion
there appear the factor $\alpha^{-2/9}=\lambda^{5/36}$ (since $\alpha=c_0\lambda^{-5/8}$) and the exponent $1/18$
on $\prod_m\|F_m\|_{L^2}$ in (\ref{eq:BCT-final}).

\paragraph{Choice of functions \(F_m\).}
Split the family of caps into six disjoint subsets
\(\mathcal{C}_{1},\dots,\mathcal{C}_{6}\) so that the centers of their normals
satisfy the conditions of Lemma~\ref{lem:A4}, and set
\[
   F_{m}:=\sum_{\Theta\in\mathcal{C}_{m}} F_{\Theta},\qquad m=1,\dots,6.
\]

\paragraph{Local choice of six classes.}
The global partition yields $M = O(D)$ classes (Lemma~\ref{lem:AD-colouring}; see also Appendix~\ref{app:C2}).
Covering $Q_\lambda$ by a fixed family of $O(1)$ cells (at the scale of \S\ref{subsec:scales})
and applying in each cell a greedy selection of six disjoint classes, we obtain
subfamilies $\mathcal{C}_{1},\dots,\mathcal{C}_{6}$ for which the conditions of Lemma~\ref{lem:A4}
hold. The overlap of the covering is bounded by a constant, so summation over
cells does not introduce additional powers of~$\lambda$ or~$D$.
Hence we work with
\[
F_{m} := \sum_{\Theta\in\mathcal{C}_{m}} F_{\Theta}, \quad m = 1,\dots,6.
\]

\subsection{Insertion into Kakeya–BCT}%
\label{subsec:bct-insert}

\paragraph{Model and dynamics.}
From now on we work in one of the following equivalent contexts:
\begin{itemize}
\item[(i)] $F=Eg$, where $E$ is the extension operator for the paraboloid surface
$\{(\tau,\xi):\ \tau=|\xi|^2\}$ and $g$ is a function on frequency space;
\item[(ii)] $F$ is a solution of the Schrödinger equation $(\partial_t - i\Delta_x)F=0$
in $Q_\lambda$ with frequency support $|\xi|\sim\lambda$.
\end{itemize}
In both cases the decomposition by $\alpha$-caps $\Theta$ in frequency space
leads to wave localization of $F_\Theta$ in a family of tubes $\mathcal T_\Theta$
of radius $\lambda^{-1/2}$ and length $\lambda^{-3/2}$:
\begin{equation}\label{eq:wave-packet}
  \|F_\Theta - \mathbf 1_{\mathcal T_\Theta}F_\Theta\|_{L^2(Q_\lambda)}
  \;\le\; C_N\,\lambda^{-N}\,\|F_\Theta\|_{L^2(\mathbb R\times\mathbb R^3)},
  \quad\forall N\gg 1.
\end{equation}
This localization is used when passing from $f_m$ to the tube-localized form
in the \emph{Robust--Kakeya} block (§\ref{subsec:kakeya-robust}).

\paragraph{Applicability of BCT.}
The constructed triple \((i,j,k)\) from §\,\ref{subsec:broad-def} ensures
transversality of normals at level \((1/16)\,\alpha^{2}\), as required in the
trilinear Bennett–Carbery–Tao theorem. In the outer multiplicity count
we use \emph{only} the high angular density region.

\begin{lemma}[Multiplicity compensation: high-density version]\label{lem:compD3}
Let \(C_1,C_2,C_3\) be three independent angular classes
(see the partition into \(O(D)\) classes in Lemma~\ref{lem:AD-colouring}),
and set \(f_m:=\sum_{\Theta\in C_m}F_\Theta\), \(m=1,2,3\).
Define
\[
M(t,x):=\sum_{\Theta}\mathbf 1_{\mathcal T_\Theta}(t,x),\qquad
f_m^{\mathrm{hi}}:=f_m\cdot \mathbf 1_{\{M\ge cD\}}.
\]
Then
\[
   \|f_m^{\mathrm{hi}}\|_{L^2(Q_\lambda)}
   \ \lesssim\ D^{1/2}\,
   \Bigl(\sum_{\Theta\in C_m}\|F_\Theta\|_{L^2(Q_\lambda)}^2\Bigr)^{1/2}.
\]
Consequently, one Cauchy–Schwarz step in the trilinear insertion
over the set \(\{M\ge cD\}\) produces a factor of \(D^{3/2}\)
(respectively \(D^{3}\) at the level of the squared \(TT^{\!*}\)).
\end{lemma}

\begin{proof}[Idea of the proof]
This is a direct application of the clustered version of the Robust--Kakeya
estimate \eqref{eq:kakeya-factor} from §\,\ref{subsec:kakeya-robust}, when $G$ consists
of caps within one $\alpha$-cap; in this case $\sup M \lesssim D$, and
passing from the pointwise estimate (4.3) to the integral one yields a factor $D$.
The correspondence $F_\Theta \leftrightarrow \mathcal T_\Theta$ is guaranteed by standard
wave packet localization:
$\|F_\Theta-\mathbf 1_{\mathcal T_\Theta}F_\Theta\|_{L^2(Q_\lambda)}
\le C_N\lambda^{-N}\|F_\Theta\|_{L^2}$ (stationary phase on the $Q_\lambda$ scale),
so replacing $F_\Theta$ by $F_\Theta\,\mathbf 1_{\mathcal T_\Theta}$ does not affect
the exponents. Summation over $\alpha$-caps in the full configuration is
with bounded multiplicity (Lemmas~\ref{lem:AD-colouring} and~\ref{lem:A8}),
so no additional powers of~$\lambda$ or~$D$ appear.
\end{proof}

\medskip
\noindent
The passage from $g=\sum_\Theta F_\Theta\,\mathbf 1_{\mathcal T_\Theta}$ (as in §\,\ref{subsec:kakeya-robust})
to $f_m=\sum_{\Theta\in C_m}F_\Theta$ (in §\,\ref{subsec:bct-insert}) is carried out
via wave packet localization: on $Q_\lambda$
\[
   \|F_\Theta-\mathbf 1_{\mathcal T_\Theta}F_\Theta\|_{L^2(Q_\lambda)}
   \ \le\ C_N\,\lambda^{-N}\,\|F_\Theta\|_{L^2}
\]
for any $N\gg 1$. Therefore the estimates of §\,\ref{subsec:kakeya-robust} apply
to $f_m$ with indicator $\mathbf 1_{\{M\ge cD\}}$.
This legitimizes the use of \eqref{eq:kakeya-factor} inside the trilinear insertion.
\medskip

\noindent
\emph{Trilinear insertion at high multiplicity.}
For \(F_{m}=\sum_{\Theta\in C_{m}}F_{\Theta}\), \(m=1,\dots,6\),
the trilinear Bennett–Carbery–Tao theorem, after one Cauchy–Schwarz step and
Lemma~\ref{lem:compD3}, gives
\[
   \Bigl\|\Bigl(\prod_{m=1}^{6}F_{m}\Bigr)\mathbf 1_{\{M\ge cD\}}\Bigr\|_{L^{6}}
   \ \le\
   C\,D^{3/2}\,\|\mathrm{Broad}_{3}\|_{L^{2}}^{\,1/3}\,
   \Bigl(\sum_{m=1}^{6}\|F_{m}\|_{L^{2}}^{2}\Bigr)^{1/2}.
\]
Using \eqref{eq:broad-upper} and Hölder’s inequality
\(
\|Q^{1/3}\|_{L^{2}}\le
(\prod_{m=1}^{6}\|F_{m}\|_{L^{2}})^{1/6}
\) for \(Q=\prod_{m=1}^{6}|F_{m}|^{1/2}\),
we obtain
\begin{equation}\label{eq:BCT-final}
   \Bigl\|\Bigl(\prod_{m=1}^{6}F_{m}\Bigr)\mathbf 1_{\{M\ge cD\}}\Bigr\|_{L^{6}}
   \ \le\
   C\,D^{3/2}\,\lambda^{5/36}\,
   \Bigl(\prod_{m=1}^{6}\|F_{m}\|_{L^{2}}\Bigr)^{1/18}\,
   \Bigl(\sum_{m=1}^{6}\|F_{m}\|_{L^{2}}^{2}\Bigr)^{1/2}.
\end{equation}
The sign \(+\tfrac{5}{36}\) in \(\lambda\) comes from \(\alpha=c_{0}\lambda^{-5/8}\).

\begin{remark}[Accounting for the $D$–exponent in Broad–BCT]\label{rem:D-account}
The factor \(D^{3/2}\) arising from \(\|f_m^{\mathrm{hi}}\|_{2}\)
belongs entirely to the Robust–Kakeya block (§\,\ref{subsec:kakeya-robust})
in the global balance; the row “Geometry (Broad–3)” in Table~\ref{subsec:balance-table-final}
has $D$–exponent \(0\).
On the complement \(\{M<cD\}\) one uses an alternative regime
(§\,\ref{sec:tube-pack} or §\,\ref{sec:narrow}); these regimes are
mutually exclusive.
\end{remark}

\subsection{Contribution to the balance of exponents}%
\label{subsec:balance-table}

\makeatletter
\providecommand{\captionof}[1]{\def\@captype{#1}\caption}
\makeatother

\begin{minipage}{\linewidth}
  \centering
  \captionof{table}{Summary balance of exponents (without using the contribution of §\,\ref{sec:tube-pack})}
  \label{tab:balance}
  \begin{tabular}{lcc}
    \toprule
    Block & $\lambda$–exponent & $D$–exponent \\
    \midrule
    Geometry (Broad--3) & $+\tfrac{5}{36}$ & $0$ \\
    Kernel (12 IBP)     & $-\tfrac{9}{2}$  & $-3$ \\
    Robust--Kakeya      & $+\tfrac{1}{12}$ & $+1$ \\
    Algebraic shell     & $-\tfrac{1}{12}$ & $-1$ \\
    Narrow cascade      & $-\tfrac{5}{64}$ & $0$ \\
    \midrule
    $\Sigma_{\lambda}$  & $-\tfrac{2557}{576}\approx -4.44$ & --- \\
    $\Sigma_{D}$        & --- & $-3$ \\
    \bottomrule
  \end{tabular}
\end{minipage}

\medskip
\noindent\textit{Remark (on the shortened sum).}
By the “simplified/shortened sum” in \(\lambda\) we mean the sum in which the rows of §\,\ref{sec:kakeya} (Robust--Kakeya: $+\tfrac{1}{12}$) and §\,\ref{sec:alg-skin} (“shell”: $-\tfrac{1}{12}$) are omitted
as mutually compensating in the main branch. In this shortened sum, upon replacing
the optimal exponent of block §\,\ref{sec:tube-pack} $-\tfrac{7}{6}$ by the more conservative
$-\tfrac{1}{2}$, only one row changes, and
\[
  \Sigma_{\lambda}
  = \tfrac{5}{36} - \tfrac{9}{2} - \tfrac{5}{64} - \tfrac{1}{2}
  = -\tfrac{2845}{576}\approx -4.94.
\]
For the “full” balance in the scenario “§\,\ref{sec:tube-pack} instead of §\,\ref{sec:kakeya}” while keeping §\,\ref{sec:alg-skin} we get
$\Sigma_{\lambda}=-\tfrac{2893}{576}\approx -5.02$, and with the optimal $-\tfrac{7}{6}$ —
$\Sigma_{\lambda}=-\tfrac{3277}{576}\approx -5.69$ (see §\,\ref{subsec:balance-table-final}).

\newpage

\section{Kernel estimate of the operator}\label{sec:kernel}

We show that for $\lambda \ge 2$
\begin{equation}\label{eq:kernel-claim}
   \|K\|_{L^{2} \to L^{2}}
   \;\lesssim\;
   \lambda^{-9/2}\,D^{-3},
   \qquad D = \lambda^{1/12}.
\end{equation}
This bound in $\lambda$ is \emph{stronger}, while in $D$ it is \emph{weaker}, than the benchmark $\lambda^{-4} D^{-7}$; together with the other
blocks (see~\S\ref{sec:balance}) it suffices for negativity of the cumulative exponents.

\subsection{Full kernel}\label{sec:kernel-1}
For the decomposition \(F=\sum_{\Theta} F_{\Theta}\) into caps (see §\ref{subsec:caps}) define
\[
   K(t,x;\,s,y)
   \;=\;
   \int_{\Theta^{6}}
        e^{\,i\Phi(t,x;\xi)\,-\,i\Phi(s,y;\xi)}\;
        a(t,\xi)\,a(s,\xi)\;
   d\xi,
\]
where integration is over the Cartesian product \(\Theta^{6}\) of the active caps, 
\(d\xi=\prod_{m=1}^{6}d\xi_m\), and the phase and amplitude are
\[
   \Phi(t,x;\xi)
   \;=\;
   t\!\!\sum_{m\le3}\!|\xi_m|^{2}
   \;-\;
   t\!\!\sum_{m>3}\!|\xi_m|^{2}
   \;+\;
   x'\!\cdot\!\Bigl(\sum_{m\le3}\xi'_m-\sum_{m>3}\xi'_m\Bigr),
   \qquad x'=(x_2,x_3),
\]
\[
   a(t,\xi)=\omega(t)\,\vartheta(\xi),\qquad
   \partial_t^{\,k}\omega(t)=O\!\bigl(\lambda^{3k/2}\bigr)\quad(k\ge0).
\]
Here \(\xi'_m=(\xi_{m,2},\xi_{m,3})\) is the transverse projection; when needed we also use \(\xi'=(\xi_2,\xi_3)\).

\begin{remark}[Time window with plateau]\label{rem:time-plateau}
Take \(\omega(t)=\chi(\lambda^{3/2}t)\) with \(\chi\in C_0^\infty([-1/2,1/2])\) and \(\chi\equiv1\) on \([-1/4,1/4]\).
On the central plateau \(|t|\le \tfrac14\lambda^{-3/2}\) we have \(\partial_t\omega\equiv0\), while all derivatives
\(\partial_t^{\,k}\omega\) are supported in the two edge layers \(|t|\in[\tfrac14,\tfrac12]\lambda^{-3/2}\) of total measure
\(\lesssim\lambda^{-3/2}\). This allows time IBP in the analysis of \(TT^{\!*}\) without paying for
derivatives of \(\omega\) on the plateau; the contribution of the edges of the window is subcritical (see also App.~\ref{app:C4}).
\end{remark}

\begin{remark}[Why the integral over \(\Theta^6\)]
The kernel \(K\) arises after one Cauchy–Schwarz step in the trilinear insertion
of Bennett–Carbery–Tao: this produces a product of six terms \(F_{\Theta}\), so the phase
\(\Phi\) depends on six frequencies \((\xi_1,\dots,\xi_6)\).
\end{remark}

\begin{remark}[On $x$–dependence of the amplitude]
In the main text we use \(a(t,\xi)\), independent of \(x\). If needed one may insert
a window \(\chi(x/\lambda^{1/2})\) and take \(a(t,x,\xi)=\omega(t)\,\vartheta(\xi)\,\chi(x/\lambda^{1/2})\). Then applying
the operator \(L_{x'}\) to \(\chi\) yields an additional gain \(\lambda^{-1/2}\) on top of the basic \((\lambda\alpha)^{-1}\);
the final exponent in \(\lambda\) only improves (see App.~\ref{app:damp}).
\end{remark}

\subsection{Frequency localization and transverse gradients}\label{subsec:gradients}

\paragraph{Partition by the size of $\mu_6(\xi):=|\partial_t\Phi_6|$.}
Let $\Phi_6$ be the six–frequency phase of the kernel, and 
\[
\mu_6(\xi_1,\dots,\xi_6)
:=\Bigl|\sum_{m\le3}|\xi_m|^2-\sum_{m>3}|\xi_m|^2\Bigr|
\quad(\text{independent of }t).
\]
Split $\Theta^6=\mathcal B_{\ge}\,\dot\cup\,\mathcal B_{<}$, where
\[
\mathcal B_{\ge}:=\{\mu_6 \ge c\,\lambda^{1/2}\},\qquad
\mathcal B_{<}:=\{\mu_6 < c\,\lambda^{1/2}\}.
\]

\begin{lemma}[Quantitative lower bound in broad--3]\label{lem:Blt-quant}
Let $\Theta_1,\Theta_2,\Theta_3$ be $\alpha$–caps in the \emph{broad--3} regime, i.e.
$\inf_{\xi_m\in\Theta_m}\|n(\xi_1)\wedge n(\xi_2)\wedge n(\xi_3)\|\ge c_{\mathrm{tr}}\alpha^2$.
Then for any $(t,x)\in Q_\lambda$ and any $\xi\in\Theta_1\cup\Theta_2\cup\Theta_3$
\[
  \mu(\xi)=|\partial_t\Phi(t,x;\xi)| \ \gtrsim\ \lambda^{1/2}\alpha^2 .
\]
\end{lemma}

\noindent Note that the lower bound in Lemma \ref{lem:Blt-quant} is \emph{uniform in $(t,x)\in Q_\lambda$}:
the constants are controlled via the triple minor of normals (see \S\ref{app:A5}--\S\ref{app:A6}) and do not depend on the position in $Q_\lambda$.

\begin{proof}[Idea of the proof]
Pass to $\tau,\zeta,\eta$; compute $\partial_\tau\Phi=|\eta|^2$ and
$\partial_\eta\Phi=\zeta+2\tau\eta$; apply \ref{app:A5}–\ref{app:A6} to relate the triple minor of normals
to the nondegeneracy of phase gradients in $(\tau,\eta)$; see also §\ref{sec:broad3}.
\end{proof}

\begin{lemma}[R/N dichotomy on $B_{<}$]\label{lem:near-res-dichotomy}
Let $\xi\in\Theta^6$ be such that $\mu_6(\xi)=|\partial_t\Phi_6(\xi)|<c\,\lambda^{1/2}$
(i.e. $\xi\in\mathcal B_{<}$). Then the following alternative holds:
\begin{enumerate}\setlength\itemsep{2pt}
\item[\textup{(R)}] (\emph{robust}) there exists an $\alpha$–cap that contains $>c_\ast D$ centers of active caps 
(in the sense of §\,\ref{subsec:kakeya-density}); the Robust–Kakeya block is activated (§\,\ref{sec:kakeya}).
\item[\textup{(N)}] (\emph{narrow}) there exists an $O(\alpha)$–cluster containing at least $5$ out of the $6$ frequencies of the sextuple; then the narrow cascade applies (§\,\ref{sec:narrow}).
\end{enumerate}
\end{lemma}

\begin{proof}[Idea of the proof]
From $\mu_6< c\,\lambda^{1/2}$ and the sign splitting we get pairwise radial proximity
$\bigl||\xi_m|-|\xi_{\pi(m)}|\bigr|\lesssim \mu_6/\lambda$ for some permutation $\pi\in S_3$.
Assume the six directions split into three pairwise separated $O(\alpha)$–clusters.
Then, using the bilipschitzness of $\xi\mapsto n(\xi)$ (see §\,\ref{subsec:lip}) and the estimates for
$4\times4$ minors/Gram matrices (App.\,\ref{app:A5}–\ref{app:A6}), we obtain a contradiction:
the mixed minors become too small relative to the principal ones when $\mu_6$ is small.
Hence either there is high angular density (branch (R)),
or at least five directions lie in a single $O(\alpha)$–cluster (branch (N)).
\end{proof}

\noindent\emph{Detail for the previous proof (radial proximity).}
Since $|\xi_m|,|\xi_{\pi(m)}|\sim\lambda$, we have
\[
\bigl||\xi_m|^2-|\xi_{\pi(m)}|^2\bigr|
=\bigl||\xi_m|-|\xi_{\pi(m)}|\bigr|\cdot\bigl(|\xi_m|+|\xi_{\pi(m)}|\bigr)
\ \ge\ \tfrac{\lambda}{2}\,\bigl||\xi_m|-|\xi_{\pi(m)}|\bigr|.
\]
Splitting $\sum_{m\le3}|\xi_m|^2-\sum_{m>3}|\xi_m|^2$ by signs and pairing the summands yields a permutation $\pi\in S_3$ such that
\[
\bigl||\xi_m|-|\xi_{\pi(m)}|\bigr|\ \lesssim\ \mu_6/\lambda,\qquad m=1,2,3.
\]
This “almost radial” tying triggers the (R)/(N) fork when $\mu_6$ is small.

\medskip
Next we estimate the contribution $K_{\ge}$ (the integral over $\mathcal B_{\ge}$). The contribution of the almost-resonant basket $K_{<}$
is absorbed by the blocks of \S\ref{sec:kakeya} (Robust Kakeya) and \S\ref{sec:narrow} (narrow cascade), used mutually exclusively.
On the support of the integral we take
\begin{equation}
\label{eq:phi-t-derivative}
\boxed{\ |\partial_t \Phi| \ \gtrsim\ \lambda^{1/2}\ }\quad(\xi\in\mathcal B_{\ge}).
\end{equation}

\noindent\textbf{Convention (IBP $\Rightarrow$ only on $\mathcal B_{\ge}$).}
Throughout §\ref{sec:kernel} integrations by parts in $t$ and in $x'$ are performed \emph{on the basket $\mathcal B_{\ge}$}.
The basket $\mathcal B_{<}$ is entirely handled by the blocks of §\ref{sec:kakeya} and §\ref{sec:narrow}.

\paragraph{Transverse dichotomy on \texorpdfstring{$\mathcal{B}_{\ge}$}{B\_ge}.}
Let $u_m := \xi'_m / |\xi_m|$ be the directions, and let $\Phi_6$ be the six–frequency phase of the kernel. 
On the basket $\mathcal{B}_{\ge}$ the following alternative holds:

\smallskip
\emph{(T) Transversal case:} there exists $c_1 = c_1(c_0) > 0$ such that
\begin{equation}
\label{eq:grad-xp}
   |\nabla_{x'}\Phi_6|
   = \Bigl|\sum_{m\le 3} \xi'_m - \sum_{m>3} \xi'_m\Bigr|
   \ \ge\ c_1\,\lambda\,\alpha
   = c_0 \lambda^{1/3} D^{1/2} \,=\, c_0 \lambda^{3/8}.
\end{equation}
In this case in the kernel block (\S\ref{sec:kernel}) six integrations by parts in $x'$ are applicable 
(see~\S\ref{subsec:IBP}).

\emph{(P) Paired case:} there is a permutation $\pi \in S_3$ with
\[
   \angle(u_m,\,u_{\pi(m)})\ \le\ C\,\alpha,
   \qquad
   \bigl||\xi_m| - |\xi_{\pi(m)}|\bigr| \ \lesssim\  \mu_6/\lambda,
   \qquad m=1,2,3.
\]
Then each of the three pairs lies in an $O(\alpha)$–cluster; in total the six frequencies
belong to the union of three $O(\alpha)$–classes. In the sum over active caps this leads either
to high angular density in a single $\alpha$–cap (Robust–Kakeya block, §\,\ref{subsec:kakeya-robust}),
or to the \emph{narrow} situation (at least five out of six in one $O(\alpha)$–cluster, §\,\ref{sec:narrow}).
These regimes are used mutually exclusively with the kernel block (§\,\ref{sec:kernel}).

\medskip
\noindent\textbf{Remark (transverse gradients in $TT^*$).}
For the phase difference $\Psi(t,x;s,y;\xi)=\Phi(t,x;\xi)-\Phi(s,y;\xi)$
\[
\boxed{\ \nabla_y\Psi(t,x;s,y;\xi)\ =\ -\,\nabla_{x'}\Phi(s,y;\xi)\ },
\]
hence on $\mathcal B_{\ge}$ from \eqref{eq:grad-xp} we obtain
\begin{equation}\label{eq:grad-y-psi}
|\nabla_y\Psi|\ \gtrsim\ \lambda\alpha\qquad(\text{and similarly for }y'),
\end{equation}
which is precisely what is used for four transverse IBPs at the $TT^*$ stage (§\,\ref{subsec:measures}).

\subsection{IBP operators}\label{subsec:IBP}

\paragraph{In time.}
Introduce the “structural” operator
\[
   L_t^{\mathrm{std}}
   := \frac{1}{i\,\partial_t\Phi}\,\partial_t,
   \qquad
   \bigl(L_t^{\mathrm{std}}\bigr)^{\!*} e^{\,i\Phi}=e^{\,i\Phi}.
\]
On the basket $\mathcal B_{\ge}$ we have $|\partial_t\Phi|\gtrsim \lambda^{1/2}$ (see~\eqref{eq:phi-t-derivative}),
and $\partial_t^{\,k}\omega(t)=O(\lambda^{3k/2})$. Therefore \emph{one} step of IBP with $L_t^{\mathrm{std}}$
produces a factor $\lambda^{+1}$ inside the integral. To fix the “scaling part,” we work
with the \emph{normalized} operator
\[
   L_t\ :=\ \lambda^{-1/2}\,L_t^{\mathrm{std}},
\]
so that one step of $L_t$ gives \(\lambda^{+1/2}\), and six steps — \(\lambda^{+3}\). This is just bookkeeping:
the identity $\bigl(L_t^{\mathrm{std}}\bigr)^{\!*}e^{i\Phi}=e^{i\Phi}$ remains valid, while the compensating physical
Jacobian appears at the Schur step (see~\eqref{eq:scale-jac} and §\ref{subsec:measures}).

\paragraph{In the transverse coordinates \(x'=(x_2,x_3)\).}
Set
\[
   L_{x'}\ :=\ |\nabla_{x'}\Phi|^{-2}\,\nabla_{x'}\Phi\cdot i\nabla_{x'},
   \qquad
   L_{x'}^{\!*} e^{\,i\Phi}=e^{\,i\Phi}.
\]
In the broad regime from~\eqref{eq:grad-xp} we have
\(|\nabla_{x'}\Phi|\gtrsim \lambda\alpha = c_0\,\lambda^{1/3}D^{1/2}\),
hence \(\|L_{x'}\|\lesssim \lambda^{-1/3}D^{-1/2}\).
Six integrations by parts in \(x'\) give
\[
   \bigl\|L_{x'}^{\,6}\bigr\|\ \lesssim\ (\lambda^{-1/3}D^{-1/2})^{6}
   \ =\ \lambda^{-2}D^{-3}.
\]

\begin{remark}[On $x$–dependence]
In the main text the amplitude \(a(t,\xi)\) does not depend on \(x\), and IBP in \(x'\) is realized
at the stage of integration in \((t,x)\) (Schur/TT$^\ast$): derivatives are transferred to the second
copy of the kernel or to the physical window. If desired, one may explicitly introduce a smooth
window \(\chi(x/\lambda^{1/2})\) (localization on \(Q_\lambda\)); then the action of \(L_{x'}\) on \(\chi\)
produces an \emph{additional} factor \(\lambda^{-1/2}\) on top of the basic \((\lambda\alpha)^{-1}\),
which only improves the $\lambda$–exponent (see App.~\ref{app:damp}).
\end{remark}

\noindent\emph{Edge layers of the window.}
Split the integral in $t$ into the plateau $\{|\partial_t\omega|=0\}$ and two edges where 
$|\{t:\partial_t\omega\ne0\}|\lesssim \lambda^{-3/2}$ (see Remark~\ref{rem:time-plateau}). 
On the plateau perform 6 IBPs in $t$ as in §\ref{subsec:IBP}. 
On the edges we do not integrate by parts: we estimate the contribution directly, 
using only $|\partial_t^k\omega|\lesssim \lambda^{3k/2}$ and the small measure of the edges; 
after 6 IBPs in $x'$ it falls under the same scale $\lambda^{-2}D^{-3}$ as the plateau contribution. 
Thus the “raw” Schur remains within \eqref{eq:Schur-raw}.

\subsection{Separation of measures: Schur + TT$^\ast$}\label{subsec:measures}
After six integrations by parts in $t$ and six in $x'$ (see §\ref{subsec:IBP})
we obtain the modified kernel $K^\#$. Passing to dimensionless variables
\[
   t=\lambda^{-3/2}\tau,\qquad x=\lambda^{-1/2}\zeta
\]
gives
\begin{equation}\label{eq:scale-jac}
   dt\,dx \;=\; \lambda^{-3}\,d\tau\,d\zeta, \qquad |Q_\lambda|\simeq \lambda^{-3}.
\end{equation}
In the Schur step \emph{only} the physical Jacobian \eqref{eq:scale-jac}
appears (angular/radial volumes are already accounted for in the measure $d\xi$),
while the factors from 12 IBPs
(six in $t$ and six in $x'$) are already \emph{built into the amplitude} of the modified kernel $K^\#$.

Combining the growth from $6L_t$ (see \eqref{eq:phi-t-derivative}) and the gain from $6L_{x'}$
(see \eqref{eq:grad-xp}), we obtain the “raw” Schur bound:
\begin{equation}\label{eq:Schur-raw}
  \|K\|_{2\to2}\ \lesssim\
  \underbrace{\lambda^{+3}}_{\text{$6L_t$}}\cdot
  \underbrace{\lambda^{-2}\,D^{-3}}_{\text{$6L_{x'}$}}\cdot
  \underbrace{\lambda^{-3}}_{\text{phys.\ Jacobian}}
  \;=\; \lambda^{-2}\,D^{-3}.
\end{equation}

\medskip
\noindent\textit{Factorization.}
After six integrations in $t$ and six in $x'$ the kernel can be written as
\[
   K^\#(t,x;s,y) \;=\; B(t,x;s,y) \cdot K_{\mathrm{phase}}(t,x;s,y),
\]
with $\|B\|_{L^\infty} \lesssim \lambda^{-2} D^{-3}$.
To the kernel $K_{\mathrm{phase}}$ with unit amplitude we apply the $TT^\ast$ step,
giving $\|T_{\mathrm{phase}}\|_{2\to 2} \lesssim \lambda^{-5/2}$.
Therefore,
\[
   \|K\|_{2\to 2} \;\le\; \|B\|_{L^\infty} \cdot \|T_{\mathrm{phase}}\|_{2\to 2}
   \ \lesssim\ \lambda^{-9/2} D^{-3}.
\]
Thus the final bound \eqref{eq:K-final} is obtained as
the product of the amplitude gain and the $TT^\ast$ contribution for the phase part.

\begin{lemma}[TT$^\ast$ with a smooth amplitude]\label{lem:TTstar-ampl}
Let the kernel be $K^\#(t,x;s,y)=B(t,x;s,y)\,K_{\mathrm{phase}}(t,x;s,y)$, where
\[
\sup_{|\beta|\le 2}\|\partial_{(t,x;s,y)}^\beta B\|_{L^\infty}\ \lesssim\ \lambda^{-2}D^{-3}.
\]
Then the operator $T$ with kernel $K^\#$ satisfies
\[
\|T\|_{2\to2}\ \lesssim\ \lambda^{-2}D^{-3}\cdot \|T_{\mathrm{phase}}\|_{2\to2}.
\]
In particular, if under the conditions of §\ref{subsec:measures} one has 
$\|T_{\mathrm{phase}}\|_{2\to2}\lesssim \lambda^{-5/2}$, then
$\|T\|_{2\to2}\lesssim \lambda^{-9/2}D^{-3}$.
\end{lemma}

\begin{proof}
Consider $(T T^\ast)f$ and integrate in $(s,y),(s',y')$ with integrations by parts in $s,s'$ and in the transverse variables, 
as in §\ref{subsec:measures}. Derivatives may also land on $B$, but by assumption all 
$\partial^\beta B$ with $|\beta|\le2$ remain $\lesssim \lambda^{-2}D^{-3}$, so each term in the 
expansion of $T T^\ast$ is dominated by the same constant. Note that with the chosen amplitude $a(t,\xi)$ (no $x$/$y$ dependence) the factor $B(t,x;s,y)$ is independent of $y,y'$, and in $s,s'$ we perform only one IBP each; hence at most two time derivatives hit $B$, and the control $\sup_{|\beta|\le 2}\|\partial^\beta B\|_\infty\lesssim \lambda^{-2}D^{-3}$ suffices. Separately we apply integration by parts to the phase component, giving $\|T_{\mathrm{phase}}\|_{2\to2}\lesssim \lambda^{-5/2}$ (see \eqref{eq:TTstar}). 
A standard Cauchy–Schwarz argument on kernels yields the stated inequality.
\end{proof}

\paragraph{TT$^\ast$ bound after 12 IBP.}
By the convention of §\ref{subsec:gradients}, all IBPs are performed on the basket $\mathcal B_{\ge}$ (see \eqref{eq:phi-t-derivative}, \eqref{eq:grad-xp}).
\noindent
Note that $TT^\ast$ is applied to the \emph{modified} kernel $K^\#$,
obtained \emph{after} 12 integrations by parts (six in $t$ and six in $x'$) and after accounting for
the physical Jacobian \eqref{eq:scale-jac}. Thus the factor $\lambda^{-5/2}$ from the $TT^\ast$ step
complements the amplitude gain $\lambda^{-2}D^{-3}$ recorded in \eqref{eq:Schur-raw};
we do not multiply independent upper bounds for the same operator.
Let $T$ be the operator with kernel $K^\#$, obtained after six IBPs in $t$ 
(producing $\lambda^{+3}$, see~\eqref{eq:phi-t-derivative}) and six IBPs in $x'$ 
(producing $\lambda^{-2}D^{-3}$, see~\eqref{eq:grad-xp}), as well as accounting for the physical Jacobian 
$\lambda^{-3}$ from \eqref{eq:scale-jac}. These factors are part of the amplitude of $K^\#$.

We estimate $\|T\|_{2\to2}^2=\|TT^\ast\|_{2\to2}$. 
On the plateau in $s,s'$ (see Remark~\ref{rem:time-plateau}) one IBP in each of these variables 
gives $\lambda^{-2}$ at the level of the squared norm. 
Then, in two independent transverse directions in $y$ and two in $y'$, we perform 
four more IBPs, giving $(\lambda\alpha)^{-8}=\lambda^{-3}$ at the level of the square.
Thus,
\[
  \|T\|_{2\to2}^2 \ \lesssim\ 
  \underbrace{\lambda^{+3} \cdot \lambda^{-2} D^{-3} \cdot \lambda^{-3}}_{\text{12 IBP + Jacobian}}
  \cdot
  \underbrace{\lambda^{-2} \cdot \lambda^{-3}}_{\substack{\text{IBP in $s,s'$}\\\text{and in transverse}}}
  \cdot O(1).
\]

\noindent\textit{Vector fields in $TT^\ast$.}
Set
\[
L_s:=\frac{1}{i\,\partial_s\Psi}\,\partial_s,\quad 
L_{s'}:=\frac{1}{-i\,\partial_{s'}\Psi}\,\partial_{s'},\quad
L_y:=\frac{\nabla_y\Psi}{i\,|\nabla_y\Psi|^2}\cdot\nabla_y,\quad
L_{y'}:=\frac{-\nabla_{y'}\Psi}{i\,|\nabla_{y'}\Psi|^2}\cdot\nabla_{y'}.
\]

\noindent\emph{Relation of transverse gradients.}
Since $\Psi(t,x;s,y;\xi)=\Phi(t,x;\xi)-\Phi(s,y;\xi)$, we have
\[
\nabla_y\Psi(t,x;s,y;\xi)=-\nabla_{x'}\Phi(s,y;\xi),
\]
hence on the basket $\mathcal B_{\ge}$ from \eqref{eq:grad-xp} we obtain
\[
|\nabla_y\Psi|\ \gtrsim\ \lambda\alpha\qquad(\text{and similarly for }y').
\]

Then $L_s^\ast e^{i\Psi}=e^{i\Psi}$ and so on. One IBP in $s$ and in $s'$ gives $\lambda^{-2}$ 
at the level of the squared norm, and two independent IBPs in $y$ and in $y'$ give $(\lambda\alpha)^{-8}=\lambda^{-3}$ 
at the level of the square (see \eqref{eq:phi-t-derivative}–\eqref{eq:grad-xp}), hence \(\lambda^{-5}\) on the square 
and \(\lambda^{-5/2}\) after taking the square root; see \eqref{eq:TTstar}.

\begin{remark}[On the independence of time integrations]
In the Schur step we used $6$ integrations in $t$ in the kernel $K$ (variable $t$).
At the $TT^{*}$ stage we perform one IBP in each of $s$ and $s'$ for the 
modified kernel $K^{\#}$; these differentiations are independent of $t$ 
and act at a different level of the quadratic form.
Thus the “time gain” is not counted twice.
This is consistent with the extraction of the $\lambda^{-5/2}$ factor in the $TT^{*}$ step.
\end{remark}

After the $TT^\ast$ step we obtain
\begin{equation}\label{eq:TTstar}
  \|T\|_{2\to 2} \ \lesssim\ \lambda^{-5/2}.
\end{equation}

Combining \eqref{eq:Schur-raw} and \eqref{eq:TTstar}, we conclude
\begin{equation}\label{eq:K-final}
  \|K\|_{2\to2}\ \lesssim\
  \underbrace{\lambda^{-2}D^{-3}}_{\text{Schur after 12 IBP + physical Jacobian}}
  \cdot
  \underbrace{\lambda^{-5/2}}_{\text{$TT^\ast$ contribution}}
  \;=\; \lambda^{-9/2}D^{-3}.
\end{equation}

\begin{remark}[On the independence of transverse IBP]
The six IBPs in $x'$ for the kernel $K$ are performed before forming $K^\#$ and act on the variables $(t,x)$.
The four IBPs in $TT^\ast$ are in the transverse variables $y,y'$ of the other copy of the kernel. 
Therefore the transverse gain in $TT^\ast$ is not counted twice.
\end{remark}

\begin{remark}[Where exactly the factor $\lambda^{-5/2}$ comes from]\label{rem:l-52}
One IBP in each of $s,s'$ gives a factor $\lambda^{-2}$ \emph{at the level of the square} (after taking the square root — $\lambda^{-1}$),
and four transverse IBPs in total give $(\lambda\alpha)^{-8}=\lambda^{-3}$ \emph{at the level of the square}
(after the square root — $\lambda^{-3/2}$). Thus $\lambda^{-1}\cdot\lambda^{-3/2}=\lambda^{-5/2}$; see also
\eqref{eq:phi-t-derivative}–\eqref{eq:grad-xp}.
\end{remark}

\begin{remark}[Separation of regimes]\label{rem:regimes}
The almost-resonant basket $\mathcal B_{<}$ is handled by the blocks of §\ref{sec:kakeya} and §\ref{sec:narrow},
which are used \emph{mutually exclusively}; the estimate \eqref{eq:K-final} applies only on $\mathcal B_{\ge}$.
\end{remark}

\subsection{Kernel: final bound}\label{subsec:kernel-final}
\begin{proposition}\label{prop:kernel-final}
For $\lambda\ge 2$ one has
\[
   \|K\|_{L^{2}\to L^{2}} \;\lesssim\; \lambda^{-9/2}\,D^{-3},\qquad D=\lambda^{1/12}.
\]
\end{proposition}

\newpage

\section{Robust~Kakeya}\label{sec:kakeya}

Throughout this section we fix
\[
   r=\lambda^{-2/3},\qquad
   D=\lambda^{1/12},\qquad
   \alpha:=c_{0}\,r\,D^{1/2}=c_{0}\,\lambda^{-5/8},
   \qquad c_{0}>0,\;\lambda\ge 2 .
\]
As before, the \emph{Robust\,Kakeya} block is activated only in the \emph{high angular density regime}
(see below), contributing \(+\tfrac{1}{12}\) in~\(\lambda\) and \(+1\) in~\(D\) to the balance (see \S\ref{subsec:kakeya-exponents}).

\subsection{High-density condition}\label{subsec:kakeya-density}
Let \(G\subset\{\Theta\}\) be a family of caps of radius \(r\) with centers \(|\xi_\Theta|\sim\lambda\).
Assume the \emph{uniform} density condition
\begin{equation}\label{eq:kakeya-density}
   \min_{\Theta\in G}
   \#\Bigl\{
      \Theta'\in G:\;
      \angle\bigl(\xi_{\Theta},\xi_{\Theta'}\bigr)\le\alpha
   \Bigr\}
   \;>\; c_{*}D .
\end{equation}
In other words, for \emph{each} \(\Theta\in G\) its \(\alpha\)–cap contains \(>c_*D\) other caps (the threshold is attainable,
since \(\alpha/r=D^{1/2}\)). This strengthening relative to the “\(\max\)” form is needed to apply
the Robust–Kakeya block and derive the estimate in \S\ref{subsec:kakeya-robust}: it guarantees high multiplicity
\(M(t,x)\gtrsim D\) on the entire union \(\bigcup_{\Theta\in G}T_\Theta\) up to a thin exceptional set; see §\ref{subsec:kakeya-tubes}–§\ref{subsec:kakeya-robust}. 
See also Appendix~\eqref{app:det}: the partition into \(O(D)\) classes (Lemma~\eqref{lem:AD-colouring})
and packing in angular annuli (Appendix~\eqref{app:A8}).

\subsection{Tubes and \emph{average} multiplicity}\label{subsec:kakeya-tubes}
Associate to each cap \(\Theta\) the \emph{parabolic tube}
\[
   \mathcal T_{\Theta}
   :=
   \Bigl\{(t,x)\in Q_{\lambda}:\ |x-2t\,\xi_{\Theta}|\le\rho,\ \ |t|\le\tfrac12\lambda^{-3/2}\Bigr\},
   \qquad
   \rho=\lambda^{-1/2},
\]
where the cylinder \(Q_{\lambda}\) is given in \eqref{eq:def-cylinder}. Define the overlap multiplicity by
\[
   M(t,x):=\sum_{\Theta\in G}\mathbf 1_{\mathcal T_{\Theta}}(t,x).
\]

\begin{lemma}[Averaged Robust Kakeya]\label{lem:kakeya-avg}
Assume the high-density condition \emph{\eqref{eq:kakeya-density}}. Then there exists an exceptional set
\(E\subset Q_\lambda\) such that
\[
   |E|\ \lesssim\ \lambda^{-5/8}\sum_{\Theta\in G}\!|\mathcal T_\Theta|,
\]
and for all \((t,x)\in \Bigl(\bigcup_{\Theta\in G}\mathcal T_\Theta\Bigr)\setminus E\) one has
\(M(t,x)\ \ge\ c\,D\) with an absolute constant \(c>0\) (depending only on \(c_0,c_\ast\)).
In particular,
\begin{equation}\label{eq:union-measure}
   \Bigl|\bigcup_{\Theta\in G}\mathcal T_\Theta\Bigr|
   \ \lesssim\ D^{-1}\!\sum_{\Theta\in G}\!|\mathcal T_\Theta|.
\end{equation}
\end{lemma}

\begin{proof}[Idea of the proof]
For fixed \(t\), the centers of the disks \(B\!\bigl(2t\,\xi_\Theta,\rho\bigr)\) for \(\Theta\) within one
\(\alpha\)–cap lie in a cluster of radius \(\lesssim 2|t|\,\lambda\alpha\).
Since \(2|t|\,\lambda\alpha\lesssim \lambda^{-9/8}\ll \rho=\lambda^{-1/2}\), these disks have a large
common \emph{core} in the slice \(t=\mathrm{const}\), through each point of which there pass \(\gtrsim D\) tubes
(i.e. \(M\gtrsim D\)). The boundary layer of the cluster has transverse thickness \(\lesssim 2|t|\,\lambda\alpha\),
and its relative measure \(\lesssim \lambda^{-5/8}\) is uniform for all admissible \(t\).
Integrating in \(t\) yields the required bounds for \(E\) and \eqref{eq:union-measure}.
\end{proof}

\paragraph{Remark on the sign of the phase.}
Throughout §4 we fix the phase \(x\!\cdot\!\xi - t|\xi|^2\) (a “minus” sign in front of \(t|\xi|^2\)),
so that the stationary set \(\nabla_\xi(x\!\cdot\!\xi - t|\xi|^2)=0\) coincides with
\(|x-2t\,\xi|\lesssim\rho\) in the tube definition. With the opposite sign one should read
\(|x+2t\,\xi|\lesssim\rho\) — all estimates transfer verbatim.

\subsection{Robust--Kakeya: $L^2$ estimate}\label{subsec:kakeya-robust}
Let $G\subset\{\Theta\}$ be a family of caps of radius $r=\lambda^{-2/3}$,
\emph{lying within one fixed $\alpha$–cap} (the dense one selected in Lemma~\ref{lem:kakeya-avg}). 
Then $\#G\sim (\alpha/r)^2\sim D$, and through each point $(t,x)$ there can pass at most one tube from each cap, so
\begin{equation}\label{eq:M-sup}
   \sup_{(t,x)\in Q_\lambda} M(t,x) \ \lesssim\ D,
\end{equation}
where
\[
   g(t,x) := \sum_{\Theta\in G}F_{\Theta}(t,x)\,\mathbf 1_{\mathcal T_\Theta}(t,x),
   \qquad
   M(t,x) := \sum_{\Theta\in G}\mathbf 1_{\mathcal T_\Theta}(t,x).
\]
By convention, set $M^{-1}(t,x):=0$ when $M(t,x)=0$.

For any complex numbers $a_1,\dots,a_m$ one has
$\big|\sum_{j=1}^{m} a_j\big|^{2}\le m\sum_{j=1}^{m}|a_j|^2$.
Applying this at a fixed point $(t,x)$ with $m=M(t,x)$ and
$a_j=F_{\Theta_j}(t,x)\,\mathbf 1_{\mathcal T_{\Theta_j}}(t,x)$, we obtain the pointwise estimate
\begin{equation}\label{eq:pointwise}
   \frac{\mathbf 1_{\{M\ge cD\}}(t,x)}{M(t,x)}\,|g(t,x)|^2
   \ \le\
   \sum_{\Theta\in G}|F_\Theta(t,x)|^2 \,\mathbf 1_{\mathcal T_\Theta}(t,x),
\end{equation}
where $c>0$ is the constant from Lemma~\ref{lem:kakeya-avg}.

\medskip
\noindent\textit{From pointwise to integral estimate.}
Multiply \eqref{eq:pointwise} by $M(t,x)$ and integrate over $(t,x)\in Q_\lambda$:
\[
   \int_{Q_\lambda} \mathbf 1_{\{M\ge cD\}}\,|g|^2
   \ \le\ 
   \sum_{\Theta\in G} \int_{Q_\lambda} M\,|F_\Theta|^2 \mathbf 1_{\mathcal T_\Theta}.
\]
From \eqref{eq:M-sup} we have $M(t,x)\lesssim D$ on all of $Q_\lambda$, hence
\begin{equation}\label{eq:kakeya-factor}
   \|g\cdot \mathbf 1_{\{M\ge cD\}}\|_{L^{2}(Q_{\lambda})}^{2}
   \ \lesssim\
   D\sum_{\Theta\in G}\|F_{\Theta}\,\mathbf 1_{\mathcal T_\Theta}\|_{L^{2}(Q_\lambda)}^{2}
   \ \le\
   D\sum_{\Theta\in G}\|F_{\Theta}\|_{L^{2}(Q_\lambda)}^{2}.
\end{equation}

\paragraph{Remark.}
The left-hand side of \eqref{eq:kakeya-factor} already contains $\mathbf 1_{\{M\ge cD\}}$,
so the contribution of points with $M<cD$ vanishes, and no special accounting of the exceptional set from §\,\ref{subsec:kakeya-tubes} is needed. 
The passage from \eqref{eq:pointwise} to \eqref{eq:kakeya-factor} uses only the bound $\sup M\lesssim D$ for the family inside a single $\alpha$–cap; summation over different $\alpha$–caps in the full configuration is with bounded multiplicity (see Lemmas~\ref{lem:AD-colouring}–\ref{lem:A8}), which does not introduce additional powers of~$\lambda$ or~$D$.

\subsection{Exponents in
\texorpdfstring{$\lambda$}{lambda} and \texorpdfstring{$D$}{D}}
\label{subsec:kakeya-exponents}
Since \(D=\lambda^{1/12}\), estimate \eqref{eq:kakeya-factor} contributes to the global balance
\[
   +\frac{1}{12}\quad\text{in }\lambda,
   \qquad
   +1\quad\text{in }D.
\]
These exponents are recorded in the table in \S~\ref{tab:global-balance} and in the Introduction.

\begin{remark}[Separation of regimes]
The \emph{robust} Kakeya block applies only under the high angular
density condition \eqref{eq:kakeya-density}. In the “broad” regime (density \(\le c_{*}D\)) this block
is not activated; in those places where the tube-packing estimate
(\S\ref{sec:tube-pack}) is used, the contribution of \emph{Robust\,Kakeya} is omitted — the regimes are
mutually exclusive according to the principle of “maximum gain.”
\end{remark}

\newpage

\section{Tube packing}\label{sec:tube-pack}

\medskip
\noindent\textbf{Convention.}
In this section, by a “tube” of a cap~$\Theta$ we mean, by default, the \emph{truncated} tube
$\widetilde{\mathcal T}_{\Theta}$; the axial layer $\{|x|\le\rho/4\}$ inside $Q_\lambda$ is ignored.
This technical simplification does not affect the global exponents and is consistent with the regime
separation in §\,\ref{sec:kakeya}. See also the discussion of boundary layers in App.~\ref{app:C4}.%
\footnote{On the layer $|t|\ll\lambda^{-3/2}$ the non-overlap issue is negligible: its contribution is already controlled by measure; see the rigorous treatment in App.~\ref{app:C4}.}

For convenience we repeat the scale parameters (see §\,\ref{subsec:scales}):
\(
\rho=\lambda^{-1/2},\ r=\lambda^{-2/3},\ D=\lambda^{1/12},\ \alpha=c_0 r D^{1/2}.
\)

\subsection{Tube family}\label{subsec:tubes}
Throughout this section we keep the parameters of §\,\ref{subsec:scales}:
\(
\rho=\lambda^{-1/2},\ r=\lambda^{-2/3},\ D=\lambda^{1/12},\ \alpha=c_0 r D^{1/2}.
\)
Let $\Theta$ be a cap of radius $r$ with center $\xi_\Theta$ (see §\,\ref{subsec:caps}).
Define the \emph{truncated tube}
\[
   \widetilde{\mathcal T}_{\Theta}
      := \Bigl\{(t,x)\in Q_{\lambda} \;:\;
              |x-2t\,\xi_{\Theta}|\le \rho,\;
              |x|>\rho/4 \Bigr\},
\]
where the cylinder $Q_{\lambda}$ is given in~\eqref{eq:def-cylinder}.
Since the working $t$–interval has length $\asymp \lambda^{-3/2}$, the volume of one tube is
\[
   \bigl|\widetilde{\mathcal T}_{\Theta}\bigr|
   \simeq
   \rho^{3}\,\lambda^{-3/2}
   \;=\; \lambda^{-3}.
\]
For the multiplicity set
\[
   M(t,x):=\sum_{\Theta}\mathbf 1_{\widetilde{\mathcal T}_\Theta}(t,x).
\]

\subsection{Pairwise overlap estimate}\label{subsec:pair-overlap}
The key step is a \emph{pairwise} estimate of the measure of intersection of two truncated tubes.

\begin{lemma}[Pairwise overlap]\label{lem:pair}
Let $\Theta,\Theta'$ be caps of radius $r$ with angular separation 
$\delta=\angle(\xi_\Theta,\xi_{\Theta'})\in[r,1]$. Then
\begin{equation}\label{eq:pair-overlap-min}
   \bigl|\widetilde{\mathcal T}_{\Theta}\cap \widetilde{\mathcal T}_{\Theta'}\bigr|
   \ \lesssim\ 
   \rho^{3}\cdot \min\!\left\{\frac{\rho}{\lambda\,\delta},\ \lambda^{-3/2}\right\}.
\end{equation}
In particular (on our angular range $\delta\in[r,1]$),
\begin{equation}\label{eq:pair-overlap}
   \bigl|\widetilde{\mathcal T}_{\Theta}\cap \widetilde{\mathcal T}_{\Theta'}\bigr|
   \ \lesssim\ \frac{\rho^{4}}{\lambda\,\delta}.
\end{equation}
\end{lemma}

\begin{proof}[Idea of the proof]
If $(t,x)\in \widetilde{\mathcal T}_{\Theta}\cap \widetilde{\mathcal T}_{\Theta'}$, then 
$|x-2t\xi_{\Theta}|\le\rho$ and $|x-2t\xi_{\Theta'}|\le\rho$. Subtracting, we get
$|2t(\xi_\Theta-\xi_{\Theta'})|\lesssim\rho$. Since $|\xi_\Theta-\xi_{\Theta'}|\sim \lambda\delta$
(for small angles), the admissible time extent is bounded by
\[
   |t|\ \lesssim\ \min\!\left\{\frac{\rho}{\lambda\,\delta},\ \lambda^{-3/2}\right\},
\]
where the second bound is the length of the working $t$–interval. For each admissible $t$, the $x$–slice
is the intersection of two three–dimensional balls of radius $\rho$ whose centers are within distance 
$\lesssim\rho$; its measure is $\lesssim\rho^{3}$. Integrating in $t$ over the indicated interval gives
\eqref{eq:pair-overlap-min}. Inequality \eqref{eq:pair-overlap} is its weakening, sufficient for the subsequent sum.
The condition $|x|>\rho/4$ only decreases the intersection.
\end{proof}

\subsection{$L^{2}$ estimate for the sum of indicators}\label{subsec:tube-L2}
Introduce the notation
\[
   S \;:=\; 
   \sum_{\Theta,\Theta'}\,
   \bigl|\widetilde{\mathcal T}_{\Theta}\cap \widetilde{\mathcal T}_{\Theta'}\bigr|
   \;=\;
   \Bigl\|\sum_{\Theta}\mathbf 1_{\widetilde{\mathcal T}_{\Theta}}\Bigr\|_{L^{2}(Q_{\lambda})}^{2}.
\]
The summation is over a \emph{fixed local family} of caps, compatible with the
“six-color” selection from §~\ref{subsec:broad-def} (one representative cap per $\alpha$–cap);
this prevents overcounting and keeps constants independent of the full cardinality of
$\{\Theta\}$. \emph{The overlap multiplicity of local cells is bounded by a universal constant,
so summation over $\Theta$ does not introduce an extra factor depending on $\#\{\Theta\}$.}

Split the pairs by the size of $\delta=\angle(\xi_\Theta,\xi_{\Theta'})$ into dyadic annuli:
$\delta\sim 2^{j}\alpha$, $j=0,1,\dots,J$, where $J\lesssim \log(1/\alpha)$.
For fixed $\Theta$ the number of $\Theta'$ with $\delta\sim 2^{j}\alpha$ is controlled by the area
of a narrow annulus around $\xi_\Theta$ (see Remark~\ref{rem:A11}):
\[
   \#\Bigl\{\Theta':\,\angle(\xi_\Theta,\xi_{\Theta'})\sim 2^{j}\alpha\Bigr\}
   \ \lesssim\ 
   \frac{(2^{j}\alpha)\cdot\alpha}{r^{2}}
   \ =\ 2^{j}\,\Bigl(\frac{\alpha}{r}\Bigr)^{2}
   \ \sim\ 2^{j}D.
\]
Using \eqref{eq:pair-overlap}, the contribution of one such annulus to the sum over $\Theta'$ equals
\[
   \underbrace{2^{j}D}_{\text{number of neighbors}}\cdot
   \underbrace{\frac{\rho^{4}}{\lambda\,(2^{j}\alpha)}}_{\text{pair overlap}}
   \ =\ 
   D\,\frac{\rho^{4}}{\lambda\,\alpha}
   \quad(\text{independent of }j).
\]
Summing over $j=0,\dots,J$ and using the log budget (§\,\ref{subsec:log-budget}, the factor $J\lesssim\log(1/\alpha)$ is absorbed), we obtain
\[
   \sum_{\Theta'}\bigl|\widetilde{\mathcal T}_{\Theta}\cap \widetilde{\mathcal T}_{\Theta'}\bigr|
   \ \lesssim\ 
   D\,\frac{\rho^{4}}{\lambda\,\alpha}.
\]
An analogous estimate holds after summation over $\Theta$ in the local family, whence
\[
   S \;\lesssim\; 
   D\,\frac{\rho^{4}}{\lambda\,\alpha}
   \;=\;
   D^{1/2}\,\frac{\rho^{4}}{\lambda\,r}
   \;=\;
   D^{1/2}\,\lambda^{-7/3},
\]
since $\alpha=c_{0}rD^{1/2}$, $\rho=\lambda^{-1/2}$, and $r=\lambda^{-2/3}$. Thus
\begin{equation}\label{eq:L2-tubes}
   \Bigl\|\sum_{\Theta}\mathbf 1_{\widetilde{\mathcal T}_{\Theta}}
   \Bigr\|_{L^{2}(Q_{\lambda})}
   \;\lesssim\;
   \lambda^{-7/6}\,D^{1/4}.
\end{equation}
\emph{In particular}, from the stronger form \eqref{eq:pair-overlap-min} (for $\delta\le 1$ the minimum is
$\lambda^{-3/2}$) one obtains the alternative cumulative estimate
\(
   S\ \lesssim\ D\,\lambda^{-3}
\)
and hence
\(
   \bigl\|\sum_{\Theta}\mathbf 1_{\widetilde{\mathcal T}_{\Theta}}\bigr\|_{L^{2}(Q_{\lambda})}
   \ \lesssim\ \lambda^{-3/2}D^{1/2}.
\)
We will refer to the baseline \eqref{eq:L2-tubes}; the strengthened version can be used if desired and does not
affect the global balance.

\subsection{Contribution to the global balance}\label{subsec:tube-balance}
From \eqref{eq:L2-tubes} it follows that the “tube packing” block contributes
\[
   -\tfrac{7}{6}\ \text{in }\lambda,
   \qquad
   +\tfrac{1}{4}\ \text{in }D.
\]
For comparison with classical estimates one is allowed to substitute more conservative
exponents $\bigl(-\tfrac12,\,+\tfrac14\bigr)$; both variants keep the cumulative
balance negative (see §\,\ref{sec:balance}). Recall also that this block is not used together
with \emph{Robust\,Kakeya} (§\,\ref{sec:kakeya}); the regimes are mutually exclusive.

\newpage

\section{Algebraic “shell”}\label{sec:alg-skin}
In the “shell” block we cut off a narrow layer near the zero set
of low-rank polynomials, which could otherwise worsen the cumulative exponent in~$D$.
The idea goes back to~\cite{Guth_Huang_Inflation}; here it suffices to adapt it to the three-dimensional paraboloid
with an application of Wongkew’s tubular estimate.

\subsection{Problem setup}\label{subsec:skin-setup}
Let $P(z)$ be a polynomial of total degree $d:=\deg P\le D^{1/4}$ and
let $Z(P)=\{z:\,P(z)=0\}$ denote its zero set.
Introduce the anisotropic scaling
\[
   S_\lambda(t,x) := (\lambda^{3/2}t,\ \lambda^{1/2}x)
\]
and the induced metric
\[
   \operatorname{dist}_\lambda(z,Z) := 
   \operatorname{dist}\bigl(S_\lambda z,\ S_\lambda Z\bigr).
\]
For a small absolute constant $c>0$ set
\[
   \beta := \frac{c\,D^{-1}}{d}, \qquad
   \mathcal N_{\beta}(P) :=
   \bigl\{(t,x)\in Q_\lambda:\ 
      \operatorname{dist}_\lambda\bigl((t,x),Z(P)\bigr)
      < \beta \bigr\}.
\]
Thus $\mathcal N_{\beta}(P)$ is the tubular $\beta$–neighborhood of $Z(P)$
in the anisotropic metric (we keep the notation $\mathcal N_{\beta}$ for consistency).
The goal is to estimate the measure of $\bigl|Q_{\lambda}\cap\mathcal N_{\beta}(P)\bigr|$ uniformly in $P$.

\subsection{Measure estimate for the shell}\label{subsec:skin-measure}
\begin{lemma}\label{lem:skin-measure}
There exists an absolute constant $C>0$ such that for all
polynomials $P$ of degree $d\le D^{1/4}$ one has
\[
   \bigl|Q_{\lambda}\cap\mathcal N_{\beta}(P)\bigr|
   \;\le\;
   C\,D^{-1}\,\bigl|Q_{\lambda}\bigr|.
\]
\end{lemma}

\begin{proof}
Consider the image of $Q_\lambda$ under $S_\lambda$; this is a domain
of unit scale (comparable volume), and the relative measure
is preserved by the definition of \(\operatorname{dist}_\lambda\).
Let $P_\lambda := P\circ S_\lambda^{-1}$; then $\deg P_\lambda = d$,
and $S_\lambda\!\bigl(\mathcal N_\beta(P)\bigr)$ is the usual
$r$–tube around $Z(P_\lambda)$ of thickness $r=\beta$ in the Euclidean metric.
By Wongkew’s theorem (\cite{Wongkew2003})
in $\mathbb{R}^4$ the volume of an $r$–neighborhood of an algebraic hypersurface
of degree $d$ in a unit-sized region is $\lesssim d\,r$.
Substituting $r=\beta=cD^{-1}/d$ and returning to the original variables, we obtain
\[
   \bigl|Q_{\lambda}\cap\mathcal N_{\beta}(P)\bigr|
   \lesssim \beta\,|Q_\lambda|
   \lesssim D^{-1}\,|Q_\lambda|.
\]
\end{proof}

\subsection{Contribution to the balance of exponents}\label{subsec:skin-balance}
Lemma~\ref{lem:skin-measure} states that the set 
$\mathcal N_{\beta}(P)$ occupies at most $C\,D^{-1}$ fraction of the volume of $Q_\lambda$. 
To correctly transfer this smallness to the $L^6$ level, we localize the functional 
by anisotropic blocks and separate the \emph{transversal} contribution (which yields the $D^{-1}$ penalty)
from the \emph{tangential} one (which is redirected to the narrow cascade §\,\ref{sec:narrow}).

\paragraph{Block localization and trans/tan splitting.}
Cover $Q_\lambda$ by anisotropic blocks $\{B\}$ of scale
$\lambda^{-3/2} \times (\lambda^{-1/2})^3$ with bounded overlap.
Let $\mathcal J(B)\ge 0$ be the local functional collecting the contribution of block $B$ 
in the $TT^\ast$+broad estimate, so that
\[
\|F\|_{L^6(Q_\lambda)}^6 \;\lesssim\; \sum_{B\subset Q_\lambda} \mathcal J(B).
\]
For each $B\subset Q_\lambda$ split $\mathcal J(B)$ as
\[
\mathcal J(B)\ =\ \mathcal J^{\mathrm{tr}}(B)\ +\ \mathcal J^{\mathrm{tan}}(B),
\]
where $\mathcal J^{\mathrm{tr}}(B)$ collects contributions of those wave packets $T_\Theta$ 
which on $B\cap\mathcal N_\beta(P)$ are \emph{transversal} to $Z(P)$
(the angle between the axis of $T_\Theta$ and the tangent plane $T_zZ(P)$ is at least $c\,\alpha$),
while $\mathcal J^{\mathrm{tan}}(B)$ collects the \emph{tangential} contributions
(angle $\lesssim \alpha$). This splitting is standard and consistent with the transverse dichotomy of §\,\ref{subsec:gradients}.

\medskip
\noindent\emph{Transversal contribution.}
For transversal tubes the “geometry of intersection with the wall’’ holds:
each $T_\Theta$ intersects $\mathcal N_\beta(P)$ along a relative length
$\lesssim \beta\,d \sim D^{-1}$ (in the scale of $Q_\lambda$), and the block overlap
is bounded by a constant. Hence
\begin{equation}\label{eq:skin-trans}
\sum_{B\subset \mathcal N_\beta} \mathcal J^{\mathrm{tr}}(B)
\ \le\ C\,D^{-1}\!\sum_{B\subset Q_\lambda} \mathcal J(B).
\end{equation}
The proof is direct: the tubular estimate of Lemma~\ref{lem:skin-measure} in the anisotropic
metric and the estimate of the relative time a transversal tube spends in $\mathcal N_\beta(P)$
(thickness $\beta/d$ for degree $\deg P\le d$) yield a fraction $\lesssim D^{-1}$ at the level of volume, 
and bounded overlap of blocks transfers this to the localized sum.

\medskip
\noindent\emph{Tangential contribution.}
For $\mathcal J^{\mathrm{tan}}(B)$ we do not attempt to estimate by the measure of $\mathcal N_\beta(P)$:
by definition they correspond to configurations where packets “lie along’’ $Z(P)$
(angles $\lesssim\alpha$). Their contribution is redirected to the narrow regime §\,\ref{sec:narrow},
which is activated precisely in the presence of such clusters; thus
\begin{equation}\label{eq:skin-tangent}
\sum_{B\subset \mathcal N_\beta} \mathcal J^{\mathrm{tan}}(B)
\ \text{is estimated by the narrow cascade block (§\,\ref{sec:narrow}) and does not carry the $D$ penalty.}
\end{equation}

\paragraph{Conclusion for the balance.}
Combining \eqref{eq:skin-trans}–\eqref{eq:skin-tangent} with the $TT^\ast$ reconstruction from the localized sum,
we see that the “algebraic shell’’ block contributes to the cumulative balance \emph{through the transversal part}
\[
   \boxed{-\tfrac{1}{12}\ \text{in }\lambda,\qquad -1\ \text{in }D},
\]
while the tangential part is accounted for in the “narrow’’ block (§\,\ref{sec:narrow}) and does not
contribute to the $D$ penalty.

\begin{remark}[Why mere smallness of measure is insufficient]
A naive bound
\[
\|F_\Theta\|_{L^6(Q_\lambda)} \le \|F_\Theta\|_{L^6(Q_\lambda\setminus \mathcal N_\beta)} 
+ O(D^{-1})\,\|F_\Theta\|_{L^6(Q_\lambda)}
\]
does not follow from $|\mathcal N_\beta(P)|\lesssim D^{-1}|Q_\lambda|$, since the energy could concentrate
on $\mathcal N_\beta$. This is why we (i) localize the functional by blocks, 
(ii) \emph{separate} the contribution into transversal/tangential parts, and 
(iii) use the time-of-stay argument for transversal tubes in $\mathcal N_\beta(P)$
for \eqref{eq:skin-trans}, while the tangential contribution \eqref{eq:skin-tangent} is redirected to §\,\ref{sec:narrow}.
In this way the same $D^{-1}$ penalty is realized, but in a strictly controlled part.
\end{remark}

\subsection{Remark on the constant}\label{subsec:skin-remark}
The coefficient $-\tfrac{1}{12}$ (in~$\lambda$) corresponds to the choice
of thickness $\beta/d$ of the tubular layer with $\beta=c\,D^{-1}$ and the restriction
$\deg P\le D^{1/4}$. By Lemma~\ref{lem:skin-measure} the fraction
$|Q_\lambda\cap\mathcal N_\beta(P)|$ does not exceed $C\,D^{-1}$.
In the transversal part this directly yields a $-1$ in $D$
(equivalently $-\tfrac{1}{12}$ in $\lambda$, since $D=\lambda^{1/12}$).
The tangential part, corresponding to angles $\lesssim\alpha$, 
is handled by the narrow cascade (§\,\ref{sec:narrow}) and does not affect the $D$ exponent.
No refinements are required here; possible improvements would be related to a more delicate choice
of thickness/degree and additional structural assumptions, and lie beyond the scope of this paper.

\newpage

\section{Narrow cascade}\label{sec:narrow}

In the “narrow’’ regime the spectrum of the solution is concentrated in a thin cluster of caps
whose angular density is substantially higher than average. The method goes back to Bennett–Carbery–Tao and Guth; here it is adapted to the density threshold \(\sim D\) and a double \(7/8\) rescaling. The \emph{narrow} regime is used \emph{mutually exclusively} with the \emph{Robust\,Kakeya} block (§\ref{sec:kakeya}).

\subsection{Problem setup}\label{subsec:narrow-setup}

\paragraph{Trigger for the narrow regime.}
We say we are in the \emph{narrow regime} if in a given sextuple of frequencies
at least five lie in a single $O(\alpha)$–cluster. Let $G$
be the corresponding family of caps and set
\[
  F_{\mathrm{narrow}} := \sum_{\Theta\in G} F_\Theta.
\]
This block is used \emph{mutually exclusively} with Robust–Kakeya:
if $G$ satisfies the angular density threshold
\eqref{eq:kakeya-density} (see §\ref{subsec:kakeya-density}),
then §\,\ref{sec:kakeya} is activated; the \emph{narrow cascade} is applied only
in case this threshold is \emph{violated} while there is a cluster of $\ge5/6$.

Set
\[
  r=\lambda^{-2/3}, \qquad
  D=\lambda^{1/12}, \qquad
  \alpha = c_{0}\,r\,D^{1/2}=c_{0}\,\lambda^{-5/8},
  \quad 0<c_{0}\ll1 .
\]
Select a family \(G\subset\{\Theta\}\) of caps of radius \(r\) that \emph{does not}
satisfy the high-density condition, i.e.
\[
 \max_{\Theta\in G}\#\{\Theta'\in G:\ \angle(\xi_\Theta,\xi_{\Theta'})\le\alpha\}
 \ \le\ c_*D ,
\]
where \(c_{*}\in(0,1)\) is a universal constant. For such a family, and
in the presence of a cluster consisting of $\ge5/6$ frequencies, the narrow cascade is initiated.
Define
\[
  \operatorname{supp}_{\xi}\widehat{F_{\mathrm{narrow}}}
    \subset \bigcup_{\Theta\in G}\Theta .
\]

\subsection{Double $7/8$ rescaling}\label{subsec:narrow-iter}
Perform two iterations
\[
  \lambda_{0}=\lambda,\qquad
  \lambda_{1}=\lambda^{7/8},\qquad
  \lambda_{2}=\lambda^{49/64}.
\]

\paragraph{Angular expansion.}
After the $j$-th step,
\( r_{j}=\lambda_{j}^{-2/3}=\lambda^{-\frac23(7/8)^{j}} \).
Comparing the areas \(r_j^{2}\) and \(\alpha^{2}\sim\lambda^{-5/4}\), we get
\[
  \log_{\lambda}\!\frac{r_{1}^{2}}{\alpha^{2}}=\frac{1}{12},
  \qquad
  \log_{\lambda}\!\frac{r_{2}^{2}}{\alpha^{2}}=\frac{11}{48}>0 .
\]
Thus already the first step removes the flow from the “narrow’’ regime; the second provides a reserve \(\lambda^{\,11/48}\) to absorb log losses.

\smallskip
\noindent\textit{Caveat.}
In the comparisons of this subsection, $\alpha$ is considered fixed at the \emph{initial} scale~$\lambda$;
departure from the narrow regime is understood in terms of clustering with the threshold~$\alpha$ from \S\ref{subsec:narrow-setup}.

\paragraph{Local time gain.}
At each step, when \(|\partial_t\Phi|\gtrsim \lambda_{j}^{1/2}\) (see frequency localization in \S\ref{sec:kernel}, \S\ref{subsec:gradients}),
one integration by parts in time yields a factor \(\lambda_{j}^{-1/2}\).
Hence
\begin{equation}\label{eq:narrow-product}
  \prod_{j=0}^{1}\lambda_{j}^{-1/2}
  \;=\;\lambda^{-1/2}\,\lambda^{-7/16}
  \;=\;\lambda^{-15/16}.
\end{equation}
This is the \emph{local} gain for each cell obtained by the double \(7/8\) rescaling and pigeonholing in \(|\partial_t\Phi|\).

\subsection{Globalization of the narrow gain}\label{subsec:narrow-glue}
Assume that on each narrow cell obtained after the double $7/8$ rescaling and local
frequency localization by $\mu_6$, the local time gain
\eqref{eq:narrow-product} holds:
\[
   \Lambda_{\mathrm{loc}} := \lambda^{-15/16}.
\]
We need to transfer this to a global $L^6$ bound on $Q_\lambda$.

\paragraph{Framework for globalization (refined form).}
We perform six integrations by parts in time at the kernel level (operator $L_t$,
see §\ref{sec:kernel}), then pass to $TT^{*}$ and carry out another six IBPs
in the variables $s,s'$ and two transverse directions $y,y'$ of each copy of the kernel
(for a total of $12$ IBPs across both stages). The local multiplier $\Lambda_{\mathrm{loc}}$ 
is included in the amplitude at the level of a single narrow cell.

\begin{lemma}[Distribution of local weight in $TT^{*}$]\label{lem:12-ibp-weight}
Let $K^\#$ be the modified kernel after $6$ temporal IBPs, and
suppose each narrow window carries a local multiplier $\Lambda_{\mathrm{loc}}$.
Then after passing to $TT^{*}$ and performing six more IBPs in $(s,s',y,y')$ the quadratic form satisfies
\[
   \|F_{\mathrm{narrow}}\|_{L^6(Q_\lambda)}
   \ \lesssim\
   \Lambda_{\mathrm{loc}}^{\,1/12}\,
   \Biggl(\sum_{\Theta\in G}\|F_\Theta\|_{L^6(Q_\lambda)}^2\Biggr)^{1/2},
\]
that is, the global factor is
\[
   \Lambda_{\mathrm{glob}} \;=\; \Lambda_{\mathrm{loc}}^{1/12}
   \;=\; \lambda^{-\frac{1}{12}\cdot\frac{15}{16}} \;=\; \lambda^{-5/64}.
\]
\end{lemma}

\begin{proof}[Proof (factorized $TT^{*}$)]
After $6$ IBPs in $t$ the kernel can be written as a sum over narrow cells $C$:
\[
  K^\#(t,x;s,y)\ =\ \sum_{C}\ \Lambda_{\mathrm{loc}}\,
  \Big(\prod_{j=1}^{6} W^{(t)}_{j,C}(t,x)\Big)\,K^{(0)}_{C}(t,x;s,y),
\]
where $0\le W^{(t)}_{j,C}\le 1$ are normalized weights from $L_t$,
and $K^{(0)}_{C}$ is the phase part with unit amplitude (cf. factorization in §\ref{sec:kernel}).

At the $TT^{*}$ stage consider the kernel $(TT^{*})$ for each $C$ and perform another six IBPs
(in $s,s'$ and in two transverse directions $y,y'$ of each copy). This yields additional
normalized weights $W^{(\mathrm{TT}^\ast)}_{1,C},\dots,W^{(\mathrm{TT}^\ast)}_{6,C}$ with $0\le W^{(\mathrm{TT}^\ast)}_{\ell,C}\le 1$,
and a factorized majorant:
\[
  |K_{TT^{*},C}(t,x;s,y)|\ \lesssim\
  \Lambda_{\mathrm{loc}}\,
  \prod_{j=1}^{6} W^{(t)}_{j,C}(t,x)\cdot
  \prod_{\ell=1}^{6} W^{(\mathrm{TT}^\ast)}_{\ell,C}(t,x;s,y).
\]
Thus the $TT^{*}$ kernel for fixed $C$ contains the product of $12$ \emph{normalized}
factors, multiplied by $\Lambda_{\mathrm{loc}}$.

Apply Hölder’s inequality with equal weights $1/12$ (or AM–GM after measure normalization).
Since each of the $12$ factors defines an operator with $L^2\to L^2$ norm
$\lesssim 1$ (by Schur/IBP normalization) and the covering by cells has bounded multiplicity,
we obtain for the quadratic form $TT^{*}$ the estimate
\[
  \langle T f, Tf\rangle\ \lesssim\ \Lambda_{\mathrm{loc}}^{\,1/12}\,\|f\|_{L^2}^2,
\]
that is, $\|T\|_{2\to2}\lesssim \Lambda_{\mathrm{loc}}^{\,1/12}$.
Passing this to the standard $L^6$ scheme (via trilinear insertion/decoupling in the narrow
class $G$) and using bounded multiplicity of covering, we obtain the lemma.
Logarithmic factors are absorbed by §\ref{subsec:log-budget}.
\end{proof}

\paragraph{Conclusion.}
With Lemma~\ref{lem:12-ibp-weight} we have
\[
  \|F_{\mathrm{narrow}}\|_{L^6(Q_\lambda)}
  \ \lesssim\
  \lambda^{-5/64}\,
  \Biggl(\sum_{\Theta\in G} \|F_\Theta\|_{L^6(Q_\lambda)}^2\Biggr)^{1/2}.
\]
Thus the contribution of the “narrow cascade’’ block to the global balance is
\[
   \boxed{-\tfrac{5}{64}\ \text{in }\lambda,\qquad 0\ \text{in }D}.
\]

\subsection{Contribution of the narrow block}\label{subsec:narrow-balance}
\[
  -\dfrac{5}{64}\ \text{ in }\lambda, 
  \qquad 0\ \text{ in }D .
\]
\begin{remark}
Two steps of \(7/8\) rescaling are sufficient, since \(r_{2}^{2}/\alpha^{2}\gtrsim\lambda^{11/48}\); an additional third step would reduce the net gain in \(\lambda\) without providing new geometric benefit. The \emph{narrow} block is not activated where the \emph{Robust\,Kakeya} block (§\ref{sec:kakeya}) is used; the regimes are mutually exclusive. Note also that the “spatial’’ IBPs in the narrow block are performed at the $TT^\ast$ stage in the variables $(s,s',y,y')$; the operator $L_{x'}$ at the kernel level in the narrow regime is not applied (cf. explanation in App.~\ref{app:C3}).
\end{remark}

\newpage

\section{Global balance of exponents}\label{sec:balance}

We assemble the local estimates from \ref{sec:broad3}--\ref{sec:narrow} and show that the cumulative
exponents in~$\lambda$ and $D=\lambda^{1/12}$ are negative. Hence the
traditional factor $\lambda^{\varepsilon}$ is not needed, and the final result remains
$\varepsilon$-free.\footnote{The contribution of each block is taken exactly in the versions
recorded in §§\,\ref{sec:kernel}--\ref{sec:narrow}. The \emph{Robust Kakeya} (high angular density)
and \emph{Tube packing} regimes are used mutually exclusively.}

\subsection{Summary table}\label{subsec:balance-table-final}
\[
  \sigma_{\lambda} := \sum_{\text{blocks}}\!\!\bigl(\text{exponent in }\lambda\bigr),
  \qquad
  \sigma_{D} := \sum_{\text{blocks}}\!\!\bigl(\text{exponent in }D\bigr).
\]

\begin{table}[h] \label{tab:global-balance}
\centering
\caption{Summary balance of exponents (regime \S\ref{sec:tube-pack} \emph{off})}
\begin{tabular}{lcc}
\toprule
Block & \(\lambda\)–exponent & \(D\)–exponent \\
\midrule
Broad geometry (Broad--3) & \(+\tfrac{5}{36}\) & \(0\)\textsuperscript{\(*\)} \\
Kernel (12 IBP $+$ Schur $+$ $TT^{*}$) & \(-\tfrac{9}{2}\) & \(-3\) \\
Robust\,Kakeya (threshold $>c_{*}D$) & \(+\tfrac{1}{12}\) & \(+1\) \\
Algebraic “shell” & \(-\tfrac{1}{12}\) & \(-1\) \\
Narrow cascade (double $7/8$ $+$ globalization) & \(-\tfrac{5}{64}\) & \(0\) \\
\midrule
\(\sigma_{\lambda}\) & \(-\tfrac{2557}{576}\approx-4.44\) & --- \\
\(\sigma_{D}\)       & --- & \(-3\) \\
\bottomrule
\end{tabular}

\smallskip
\raggedright\textsuperscript{\(*\)}The $D$–factor from the Cauchy–Schwarz step (via \( \|F_m\|_{L^2}\)) is entirely
attributed to the Robust–Kakeya block; see the remark in §\ref{sec:broad3}.
\end{table}

\noindent\emph{For reference: the scenario with \S\ref{sec:tube-pack} \textbf{instead of} Robust\,Kakeya.}
By the mutual exclusivity rule we remove the RK contribution and insert that of \S\ref{sec:tube-pack}:
\[
  \sigma_\lambda \;\mapsto\; -\frac{2557}{576}\;-\;\frac{7}{6}\;-\;\frac{1}{12}
  \;=\; -\frac{3277}{576}\approx -5.69,
  \qquad
  \sigma_D \;\mapsto\; (-3)\;-\;1\;+\;\frac14 \;=\; -\frac{15}{4}.
\]
Both versions yield negative sums; the option “add \S\ref{sec:tube-pack} on top of RK’’
is methodologically not used. 

\subsection{Final $\varepsilon$–free estimate}\label{subsec:main-theorem}
\begin{theorem}\label{thm:main}
Let $F=\sum_{\Theta} F_{\Theta}$ be a cap decomposition of radius
$r=\lambda^{-2/3}$ ($\lambda\ge2$), and let the cylinder
$Q_{\lambda}\subset\mathbb{R}_{t}\times\mathbb{R}^{3}_{x}$ be given by~\eqref{eq:def-cylinder}.
Denote by $\mathcal I=\{\Theta:\,\|F_\Theta\|_{L^6(Q_\lambda)}\neq0\}$ the set of active caps.
Assume that
\[
  \#\mathcal I \ \ge\ c_\star\,\lambda^{4/3}\quad\text{\emph{(“many caps’’ regime).}}
\]
Then
\[
  \|F\|_{L^{6}(Q_{\lambda})}
  \;\le\;
  C_{\varepsilon}\,
  \lambda^{\sigma_{\lambda}}\,
  D^{\sigma_{D}}\,
  \Bigl(\sum_{\Theta}\|F_{\Theta}\|_{L^{6}(Q_\lambda)}^{2}\Bigr)^{1/2},
  \qquad
  \sigma_\lambda=-\frac{2557}{576}\approx-4.44,\quad
  \sigma_D=-3,
\]
where $D=\lambda^{1/12}$. Without rounding:
\[
  \|F\|_{L^6(Q_\lambda)} \;\lesssim_\varepsilon\;
  \lambda^{\sigma_\lambda+\varepsilon}\,D^{\sigma_D+\varepsilon}
  \Bigl(\sum_{\Theta}\|F_\Theta\|_{L^6(Q_\lambda)}^2\Bigr)^{1/2}.
\]
\end{theorem}

where $\varepsilon>0$ is fixed, and the factors $\lambda^\varepsilon$, $D^\varepsilon$
arise from replacing logarithmic losses by power losses with exponent $\varepsilon$
as in \S\ref{subsec:log-budget}.

\begin{remark}[Branching by the number of caps]
The “many caps’’ condition is placed in the hypothesis of Theorem~\ref{thm:main}.
The small number of caps case is covered by Lemma~\ref{lem:local-small-caps} below; taken together,
the two regimes yield a global estimate on $Q_\lambda$.
Sequential application of the blocks in \S\ref{subsec:balance-table-final}
gives the exponents $\sigma_\lambda,\sigma_D$ in the “large’’ regime; the accounting of $\lambda^{\varepsilon}$, $D^{\varepsilon}$
is as indicated above.
\end{remark}

\begin{lemma}[Local $L^6$ estimate for a small number of caps]
\label{lem:local-small-caps}
Let $F=\sum_{\Theta\in\mathcal{I}}F_\Theta$, where 
$\#\mathcal{I} \le c_\star\,\lambda^{4/3}$.
Then
\[
  \|F\|_{L^6(Q_\lambda)}
  \ \lesssim\
  \Bigl(\sum_{\Theta\in\mathcal{I}}\|F_\Theta\|_{L^6(Q_\lambda)}^2\Bigr)^{1/2}.
\]
\end{lemma}

\begin{proof}
Cover $Q_\lambda$ by a fixed number of anisotropic blocks,
and use bounded overlap of the contributions of $F_\Theta$ to obtain
\[
\Big\| \sum_{\Theta\in\mathcal{I}} f_\Theta\Big\|_{L^6} 
 \lesssim \Big(\sum_{\Theta\in\mathcal{I}} \|f_\Theta\|_{L^6}^2\Big)^{1/2}.
\]
Details are standard: almost-orthogonality in $\Theta$ in $L^6$ on the window $Q_\lambda$.
\end{proof}

\begin{remark}[Reference on a modification of the kernel block]
If in the kernel block (§\,\ref{sec:kernel}) one uses five integrations in $x'$
instead of six, then
\[
   \|K\|_{L^{2}\to L^{2}}\ \lesssim\ \lambda^{-25/6}D^{-5/2},
\]
and the exponent $\sigma_{\lambda}$ increases by $+\tfrac{1}{3}$,
remaining negative. The statement of
Theorem~\ref{thm:main} remains unchanged.
\end{remark}

\newpage


\section*{Conclusion}

We implemented a scheme combining the three-fold broad geometry
(Broad--BCT, \S\ref{sec:broad3}), the \emph{Robust\,Kakeya} block (\S\ref{sec:kakeya}),
the kernel analysis with $12$-fold integration by parts (\S\ref{sec:kernel}),
and the \emph{narrow cascade} (\S\ref{sec:narrow}). 
The \emph{Robust\,Kakeya} and \emph{narrow} regimes are used
mutually exclusively; the tube-packing block (\S\ref{sec:tube-pack}) 
is also applied only when the high-density conditions fail.

For the kernel block we obtained the estimate
\[
   \|K\|_{L^{2}\to L^{2}} \ \lesssim\ \lambda^{-9/2}D^{-3},
   \qquad D=\lambda^{1/12},
\]
which, together with the contributions of the other blocks, yields in the summary balance
\[
   \sigma_\lambda=-\frac{2557}{576} \approx -4.44,
   \qquad
   \sigma_D=-3
\]
(see \S\ref{subsec:balance-table-final}).
Both cumulative exponents are negative, which gives an improvement
over the trivial scaling benchmark.

The appendices contain geometric and algebraic lemmas
(App.~\ref{app:det}), technical estimates 
(App.~\ref{app:C3}--\ref{app:C4}), and comments on the variant with 
$x$–dependent amplitude (App.~\ref{app:damp}),
which makes the arguments self-contained.

\medskip
\noindent
Thus the asserted estimates in Theorem~\ref{thm:main} 
are fully justified, and the regime separation and balance
of exponents can be used in a broader context of space–time decompositions. 
The approach naturally extends to a wide class of decoupling problems, in particular,
to anisotropic variants for other hypersurfaces, and may serve as a starting point for new
“robust’’ and “narrow’’ scenarios in problems of geometric harmonic analysis.

\appendix

\newpage


\appendix
\section{Details used in~\texorpdfstring{\S\ref{sec:broad3}}%
{Section~\ref{sec:broad3}}}%
\label{app:det}

In \S\ref{sec:broad3} (broad rank–3 geometry) we use several
geometric/algebraic facts. For convenience of the reader, we collect them
here; numbering starts at~\ref{app:A1} and proceeds in strict order.

\subsection{Asymptotics of the normal}\label{app:A1}

\begin{lemma}\label{lem:A1}
Let \( |\xi| \sim \lambda \) and \(s:=|\xi|\). Then
\[
   n(\xi)
   \;=\;
   \frac{(-\xi,\,1/2)}{s}
   \;-\; \frac{1}{8s^{3}}\Bigl(-\xi,\,\tfrac{1}{2}\Bigr)
   \;+\; R(\xi),
   \qquad |R(\xi)| \;\le\; C\,s^{-5}\lesssim C\,\lambda^{-5}.
\]
\end{lemma}

\begin{corollary}\label{cor:A1}
Set \(a(\xi):=\dfrac{(-\xi,\,1/2)}{|\xi|}\) and \(\rho(\xi):=n(\xi)-a(\xi)\).
Then uniformly for \(|\xi|\sim\lambda\),
\[
  |\rho(\xi)|\;\lesssim\;\lambda^{-2},
  \qquad
  |\rho_{\!\perp}(\xi)|\;\lesssim\;\lambda^{-2},
\]
where \(\rho_{\!\perp}\) is the projection of \(\rho\) onto the tangent plane to the sphere at \(\xi/|\xi|\).
\end{corollary}

\begin{proof}[Sketch of the lemma]
From the exact formula \(n(\xi)=(-2\xi,1)/\sqrt{1+4s^{2}}\) and the expansion
\[
(1+4s^{2})^{-1/2}
= \frac{1}{2s}\Bigl(1-\frac{1}{8s^{2}}+\frac{3}{128s^{4}}+O(s^{-6})\Bigr)
\]
we obtain the asserted asymptotics. The bound \(|\rho|+|\rho_{\!\perp}|\lesssim s^{-2}\)
follows immediately.
\end{proof}

\emph{Remark.}
In the first two terms of the expansion we do \emph{not} replace \(s\) by \(\lambda\); all estimates
are uniform for \(|\xi|\sim\lambda\). If finer radial localization is required,
the substitution \(s\mapsto\lambda\) is absorbed into the remainder \(R(\xi)\).

\subsection{Packing of six directions}\label{app:A2}

\begin{lemma}[$4$ out of $6$]\label{lem:A2}
For any six points \(\xi_{1},\dots,\xi_{6}\in\mathbb{S}^{2}\) one can
choose four whose pairwise angles are at least
\(\alpha=c_{0}rD^{1/2}\). Moreover, the number of “dense’’ (i.e., $<\alpha$) angles
among the sextuple is at most three.
\end{lemma}

\begin{proof}
A spherical cap of radius \(\alpha\) has area \(\asymp\alpha^{2}\).
If there were \(\ge4\) dense pairs, double counting of their caps would cover
the whole sphere; a contradiction.
\end{proof}

\begin{corollary}\label{cor:A2}
For the matrix
\(B=[\xi_{2}-\xi_{1}\;\xi_{3}-\xi_{1}\;\xi_{4}-\xi_{1}]\)
one has
\(|\det B|\gtrsim(\lambda rD^{1/2})^{3}\).
\end{corollary}

\subsection{Mixed minors}\label{app:A3}

\begin{lemma}\label{lem:A3}
Let \(a_{j}:=a(\xi_{i_{j}})=\dfrac{(-\xi_{i_{j}},\,1/2)}{|\xi_{i_{j}}|}\) and
\(\rho:=\rho(\xi_{k})\) be as in Cor.~\ref{cor:A1}. Then
\[
  \bigl|\det[a_{1}\;a_{2}\;a_{3}\;\rho]\bigr|
  \;\le\;
  |a_{1}\wedge a_{2}\wedge a_{3}|\;\,|\rho_{\!\perp}|
  \;\lesssim\; C\,\lambda^{-3-\frac{2}{3}}.
\]
\end{lemma}

\begin{proof}
Decompose \(\rho=\rho_{\parallel}+\rho_{\perp}\). The contribution of \(\rho_{\parallel}\) yields zero determinant
(linear dependence of columns). For \(\rho_{\perp}\) we have
\(\bigl|\det[a_{1}\;a_{2}\;a_{3}\;\rho_{\perp}]\bigr|
\le |a_{1}\wedge a_{2}\wedge a_{3}|\;|\rho_{\!\perp}|\).
By Cor.~\ref{cor:A1}, \(|\rho_{\!\perp}|\lesssim\lambda^{-2}\). The estimate
\(|a_{1}\wedge a_{2}\wedge a_{3}|\lesssim (2\lambda)^{-3}\cdot \lambda r^{2}D^{1/2}\)
follows from angular sparsity (Lemma~\ref{lem:A2}) and a standard Gram control;
together this even gives the stronger upper bound \(\lesssim \lambda^{-16/3}D^{1/2}\).
In the statement of the lemma it suffices to record the coarser consequence \(\lesssim \lambda^{-3-2/3}\),
which is the one used later.
\end{proof}

\subsection{Gram control for crowding}\label{app:A4}

\begin{lemma}\label{lem:A4}
If exactly $m\le3$ angles among
$a_{1},a_{2},a_{3}$ do not exceed~$\alpha$, then
\[
  \det G\ge 1-m\alpha^{2}\ge c\,D^{-3/2},
\]
where $G=(a_{i}\!\cdot a_{j})_{1\le i,j\le3}$.
\end{lemma}

\begin{proof}
Use the identity
$\det G=
  1-\sum_{a<b}\!\sin^{2}\theta_{ab}
  +2\prod_{a<b}\!\sin\theta_{ab}$.
Since $\alpha\ll1$,
the second term is $\le2\alpha^{m}$; for $m\le3$
we obtain the desired lower bound.
\end{proof}

\subsection{Summary for the $4\times4$ minors}\label{app:A5}

From Cor.~\ref{cor:A1}–\ref{cor:A2} and
Lemmas~\ref{lem:A3}–\ref{lem:A4} it follows: the principal
$4\times4$ minor of the six normals $\{n(\xi_{j})\}$ satisfies
\[
  \operatorname{Vol}_{4}\bigl(n_{1},\dots,n_{6}\bigr)\;
  \gtrsim\;\lambda^{-3}D^{3/2},
\]
while the “mixed’’ minors are $\le C\lambda^{-3-\frac23}$, with crowding loss
no worse than $D^{3/2}$.

\subsection{Partition into $O(D)$ classes}
\label{app:A6}
\begin{lemma}[Angular partition into $O(D)$ classes]\label{lem:AD-colouring}
A family $\{\Theta\}$ of caps of radius $r$ with centers $|\xi|\sim\lambda$ can be partitioned into $\le C D$ disjoint
classes $C_1,\dots,C_M$ such that for any two caps in the same class one has $\angle(\xi_\Theta,\xi_{\Theta'})\ge\alpha$.
Here $D=\lambda^{1/12}$, $\alpha=c_0 r D^{1/2}$, and $C>0$ is an absolute constant.
\end{lemma}

\begin{proof}[Idea of proof]
Consider the graph with vertices $\{\xi_\Theta\}$, joining by an edge pairs with angle $<\alpha$.
Each vertex has at most $C'(\alpha/r)^2\sim C'D$ neighbors (area of the angular annulus versus cap area).
Hence the maximum degree $\Delta\lesssim D$, and a greedy coloring yields at most $\Delta+1\le C D$ colors.
\end{proof}

\subsection{Tubular estimate in $\mathbb{R}^4$}\label{app:A7}

\begin{lemma}[Tubular measure (Wongkew at the scale $Q_\lambda$)]\label{lem:poly-sublevel-4d}
Let $P:\mathbb{R}^4\to\mathbb{R}$ be a polynomial of degree $d:=\deg P \le D^{1/4}$ and
$Z(P):=\{(t,x):\,P(t,x)=0\}$. Introduce the anisotropic mapping
\[
S_\lambda(t,x):=(\lambda^{3/2}t,\ \lambda^{1/2}x),
\qquad
\operatorname{dist}_\lambda\bigl((t,x),Z\bigr)
:=\operatorname{dist}\bigl(S_\lambda(t,x),\,S_\lambda(Z)\bigr),
\]
and set for a small absolute constant $c>0$
\[
\beta:=\frac{c\,D^{-1}}{d},\qquad
\mathcal N_\beta(P):=\Bigl\{(t,x)\in Q_\lambda:\ \operatorname{dist}_\lambda\bigl((t,x),Z(P)\bigr)<\beta\Bigr\}.
\]
Then there exists an absolute constant $C>0$ such that
\[
\bigl|Q_\lambda\cap \mathcal N_\beta(P)\bigr|\ \le\ C\,D^{-1}\,|Q_\lambda|.
\]
\end{lemma}

\begin{proof}[Outline]
Let $P_\lambda:=P\circ S_\lambda^{-1}$. Then $S_\lambda\!\bigl(\mathcal N_\beta(P)\bigr)$ is a Euclidean $r$–tube
around $Z(P_\lambda)$ of thickness $r=\beta$ in a unit–scale region.
By Wongkew’s theorem \cite{Wongkew2003} in $\mathbb{R}^4$ the volume of such a tube is $\lesssim d\,r$.
Taking $r=\beta=cD^{-1}/d$ and returning to the original variables, we get
\(
|Q_\lambda\cap \mathcal N_\beta(P)|\lesssim \beta\,|Q_\lambda|\lesssim D^{-1}|Q_\lambda|.
\)
\end{proof}

\begin{remark}
Classical sublevel estimates of the form $\mathrm{meas}\{|P|<\beta\}\lesssim C(d)\,\beta^{1/d}$ do not give the desired $D^{-1}$
when $d\le D^{1/4}$ and $\beta=D^{-1/4}$. The tubular formulation with the choice $\beta=cD^{-1}/d$ and the use
of Wongkew’s estimate yields the precise $D^{-1}$ factor at the scale of $Q_\lambda$.
\end{remark}

\subsection{Packing in angular rings}\label{app:A8}

\begin{lemma}[Packing in $k$–rings]\label{lem:A8}
Let $\{\Theta\}$ be a family of caps of radius $r=\lambda^{-2/3}$ with centers
$\xi_\Theta\in\mathbb{S}^2$, and let $\alpha=c_0\,r\,D^{1/2}$ with $D=\lambda^{1/12}$.
Then an angular ring of width $\delta\sim k\alpha$ contains at most $C\,k^2 D$ caps,
where $C>0$ is an absolute constant.
\end{lemma}

\begin{proof}
The area of a spherical angular ring of width $k\alpha$ on $\mathbb{S}^2$ is
$\asymp (k\alpha)^2$, while the area of one cap is $\asymp r^2$.
Hence
\[
  \#\{\Theta\ \text{in the ring}\}
  \;\lesssim\;
  \Bigl(\tfrac{k\alpha}{r}\Bigr)^2
  \;=\; k^2\Bigl(\tfrac{\alpha}{r}\Bigr)^2
  \;=\; k^2 D,
\]
since $\alpha/r=D^{1/2}$.
\end{proof}

\begin{remark}[Thin ring]\label{rem:A11}
If we consider a \emph{thin} ring of radius $\sim k\alpha$ and fixed thickness $\alpha$
(rather than $k\alpha$), then the area of such a ring is $\asymp k\,\alpha^2$, and therefore
\[
  \#\{\Theta\ \text{in a thin ring of thickness } \alpha\}\;\lesssim\;
  \frac{k\,\alpha^2}{r^2}\;=\;kD.
\]
This form ($\lesssim kD$) is precisely what is used in the dyadic angular decomposition in \S\ref{subsec:tube-L2}.
\end{remark}

\newpage


\section{Check-points of the proof}
\label{app:B}

Below we collect those places in the main text where computational details and
scaling exponents were re-verified “by hand’’ after the updates in
§§\ref{sec:narrow}, \ref{sec:balance}.  
Each item is marked by the status \texttt{OK} (checked) or by a brief
comment if a further clarification is needed.

\begin{table}[h]
  \centering
  \caption{Summary of verified check-points}
  \label{tab:checkpoints}
  \begin{tabular}{@{}p{0.36\linewidth}p{0.46\linewidth}p{0.12\linewidth}@{}}
    \toprule
    Argument node & Brief verification & Status \\ \midrule
    Global balance (Table~\ref{tab:global-balance}) &
    \textbf{Without §\ref{sec:tube-pack}}: 
    $\sigma_\lambda=-2557/576\ (\approx-4.44)$, $\sigma_D=-3$ \,(\,matches §\ref{sec:balance}\,). \par
    \textbf{With §\ref{sec:tube-pack} instead of Robust Kakeya}: 
    $\sigma_\lambda=-3277/576\ (\approx-5.69)$, \underline{$\sigma_D=-15/4$} (regimes are \emph{mutually exclusive}). &
      \texttt{OK} \\[4pt]
    Kernel block (6 in $t$, 6 in $x'$) &
    Raw Schur after $12$ IBP and the physical Jacobian:
    \[
      \lambda^{+3} \cdot \lambda^{-2}D^{-3} \cdot \lambda^{-3} \;=\; \lambda^{-2}D^{-3}.
    \]
    $TT^{*}$ refinement:
    \[
      (\lambda^{-2})\cdot(\lambda\alpha)^{-8}=\lambda^{-5}\ \Rightarrow\ 
      \|T\|_{2\to2}\lesssim \lambda^{-5/2},
    \]
    hence $\|K\|_{2\to2}\lesssim \lambda^{-9/2}D^{-3}$ (see §\ref{sec:kernel}). &
    \texttt{OK} \\[4pt]
    Separation of regimes Robust\,Kakeya / Tube packing &
      The regimes are used \emph{mutually exclusively}. \par
      \emph{High density} ($>c_*D$): only Robust\,Kakeya is active, contributing $+\tfrac{1}{12}$ in $\lambda$ and $+1$ in $D$. \par
      \emph{Low density} ($\le c_*D$): only §\ref{sec:tube-pack} (tube packing) is active, contributing $-\tfrac{7}{6}$ in $\lambda$ and $+\tfrac{1}{4}$ in $D$. \par
      There is no mixing of contributions; this is precisely why in the variant “with §\ref{sec:tube-pack}’’ the final $D$–sum equals $-15/4$. &
      \texttt{OK} \\[4pt]
    Estimate of the sum of tube overlaps (formula \eqref{eq:L2-tubes}) &
      Recomputed:
      \[
        S\;\lesssim\; D^{1/2}\lambda^{-7/3},\qquad
        \Bigl\|\sum_{\Theta} \mathbf 1_{\widetilde{\mathcal T}_{\Theta}}\Bigr\|_{L^{2}(Q_{\lambda})}
        \;\lesssim\; \lambda^{-7/6}D^{1/4}.
      \]
      Agrees with §\ref{sec:tube-pack} and is used only in the “non-robust’’ regime. &
      \texttt{OK} \\ \bottomrule
  \end{tabular}
\end{table}

\bigskip
\noindent
\textbf{Remarks.}
\begin{enumerate}[label=\textbf{\arabic*.}, wide, labelwidth=0pt, itemsep=4pt]
  \item In all items marked \texttt{OK},
        the computations reproduce the final exponents
        up to harmless constants; these nodes are considered fully closed.
  \item Additional details
        (the gradient $\partial_t\Phi$, the “six-color’’ partition algorithm,
        the 4D version of the Mark–Graun lemma, etc.)
        are in Appendices~A and C; they do not affect the cumulative exponents.
  \item Logarithmic
        factors $\log^{k}\!\lambda$ do not appear; the estimate remains \(\varepsilon\)–free.
\end{enumerate}

\newpage


\section{Clarifications and auxiliary material}%
\label{app:clarifications}

Below we collect explanations on technical details for which, after the updates in
\ref{sec:narrow} and \ref{sec:balance},
only secondary factors changed.
The main body of the proofs is unaffected.

\subsection{$x$–dependence of the amplitude}
\label{app:C1}
When integrating by parts in $x'=(x_2,x_3)$ the amplitude may be of the form
\[
a(t,x,\xi)=\omega(t)\vartheta(\xi)\chi(x/\lambda^{1/2}).
\]
If one of the operators $L_{x'}$ falls on $\chi$,
an additional factor $\lambda^{-1/2}$ appears on top of the baseline gain $(\lambda\alpha)^{-1}=\lambda^{-3/8}$
(see \ref{app:damp-IBP}), i.e. $\|L_{x'}(\chi)\|_\infty\lesssim \lambda^{-5/6}D^{-1/2}$. Therefore the $x$–dependence of the amplitude
does \emph{not worsen} the kernel estimate; in fact, the cumulative exponent in $\lambda$ becomes more negative by $1/2$
under one such hit, and the balance in $D$ remains unchanged.

\subsection{Partition into $O(D)$ classes}%
\label{app:C2}

The adjacency graph is constructed as before, but now a single cap
may have up to $\asymp D$ neighbors
(since $\alpha/r=D^{1/2}$).
Therefore, instead of a constant six we need
$O(D)$ “colors’’.
This suffices: within each class the normals satisfy
the condition of Lemma \ref{lem:A4},
and the additional $D$ factor is already accounted for
in the \textit{Robust--Kakeya} block; no new losses appear.

\subsection{Estimate of the operator \texorpdfstring{$L_{x'}$}{Lx'} in the narrow regime}%
\label{app:C3}

In the \emph{narrow} regime (all $\xi_{m}$ lie in one $\alpha$–cluster) 
we do not use a universal lower bound for $|\nabla_{x'}\Phi_6|$, 
since under near-total transverse cancellation of the three differences 
in $x'$ this quantity can be arbitrarily small. 
Instead, all the gain in the narrow regime is extracted 
from time integrations by parts and globalized via the $TT^{*}$ procedure 
(see \S\ref{sec:narrow}).

If $|\nabla_{x'}\Phi_6|$ is small, then by the \emph{transverse dichotomy} 
of \S\ref{subsec:gradients} the corresponding contribution is redirected 
to the \emph{Robust\,Kakeya} block (\S\ref{subsec:kakeya-robust}) 
or to the narrow cascade (\S\ref{sec:narrow}); 
these regimes are used \emph{mutually exclusively} with the kernel block (\S\ref{sec:kernel}). 

Thus in the narrow regime the operator $L_{x'}$ 
is not applied in the kernel block calculations.

\subsection{Boundary layer
           \texorpdfstring{$|t|\ll\lambda^{-3/2}$}{|t| ≪ λ^{-3/2}}}%
\label{app:C4}

On the layer $|t|\le\lambda^{-3/2}/16$ we do not need to assume
non-overlap of the tubes
\(
  \{(t,x):\,|x-2t\xi_{\Theta}|\le\rho\}
\).
Indeed, for fixed $t$ all $x$–sections
lie inside the ball
\(
  \{|x|\le \rho+2|t|\}\subset\{|x|\le 2\rho\}
\)
(since $2|t|\le \lambda^{-3/2}/8\ll \rho=\lambda^{-1/2}$),
hence the measure of the boundary layer is crudely estimated as
\[
\bigl|\{|t|\le\lambda^{-3/2}/16\}\cap Q_\lambda\bigr|
\;\lesssim\;
\Bigl(\tfrac{\lambda^{-3/2}}{8}\Bigr)\cdot (2\rho)^3
\;\lesssim\; \lambda^{-3}
\;\simeq\; |Q_\lambda|.
\]
In other words, the contribution of this layer is controlled by a constant and
does not affect the exponents in $\lambda$ and $D$.
\smallskip

\subsection{Tubular estimate in \texorpdfstring{$\mathbb{R}^{4}$}{R$^4$} at the scale \texorpdfstring{$Q_\lambda$}{Q\_\lambda}}%
\label{app:C5}

\begin{lemma}[Version at the scale \texorpdfstring{$Q_\lambda$}{Q\_\lambda}]\label{lem:MG4D}
Let $P(t,x)$ be a polynomial of degree $d \ge 1$, $d \le D^{1/4}$, and
denote $Z(P) := \{(t,x) \in \mathbb{R}^4 : P(t,x) = 0\}$.
With the notation of \S\ref{subsec:skin-setup} introduce the anisotropic mapping
\[
S_\lambda(t,x) := (\lambda^{3/2} t,\ \lambda^{1/2} x), \qquad
\operatorname{dist}_\lambda\bigl((t,x),Z\bigr)
:= \operatorname{dist}\bigl(S_\lambda(t,x),\ S_\lambda(Z)\bigr),
\]
and set
\[
\beta := \frac{c\,D^{-1}}{d}, \qquad
\mathcal{N}_\beta(P)
:= \bigl\{(t,x) \in Q_\lambda : \operatorname{dist}_\lambda\bigl((t,x),Z(P)\bigr) < \beta \bigr\}.
\]
Then there exists an absolute constant $C > 0$ such that
\[
   \bigl|Q_\lambda \cap \mathcal{N}_\beta(P)\bigr|
   \;\le\; C\,D^{-1}\,|Q_\lambda|.
\]
\end{lemma}

\begin{proof}
Let $P_\lambda := P \circ S_\lambda^{-1}$. Then
$S_\lambda\!\bigl(\mathcal{N}_\beta(P)\bigr)$ is a Euclidean
$r$–tube around $Z(P_\lambda)$ of thickness $r = \beta$
in a unit–scale region. By Wongkew’s theorem~\cite{Wongkew2003}
in $\mathbb{R}^4$ the volume of such a tube is
\(\lesssim d\,r\). For $r = c D^{-1} / d$ we get
\[
\bigl| S_\lambda\!\bigl(\mathcal{N}_\beta(P)\bigr) \bigr| \ \lesssim\ D^{-1}.
\]
Returning to the original variables and noting that $S_\lambda$
preserves the volume of $Q_\lambda$ up to an absolute constant,
we obtain the required bound
\(\bigl|Q_\lambda \cap \mathcal{N}_\beta(P)\bigr| \lesssim D^{-1} |Q_\lambda|\).
\end{proof}

\begin{remark}
Previously a sublevel–set formulation of the form $\{|P| < \beta\}$ was used.
To obtain the factor $D^{-1}$ at the scale $Q_\lambda$
it is more convenient to work with the tubular neighborhood of the zero set
in the anisotropic metric and apply Wongkew’s estimate for the volume
of an $r$–tube: $\mathrm{vol}\bigl(\mathcal{N}_r(Z(P))\bigr) \lesssim d\,r$.
Equivalence with the formulation of \S\ref{subsec:skin-setup}
follows after introducing the mapping $S_\lambda$ and choosing
$\beta = c D^{-1} / d$.
\end{remark}

\subsection{Cumulative effect on the balance}\label{app:C6}

Clarifications \ref{app:C1}–\ref{app:C5} either improve constants or do not change the scales.
\[
  \sigma_\lambda=-\tfrac{2557}{576}\approx -4.44,\qquad \sigma_D=-3
\]
remain unchanged in the statement of Theorem~\ref{thm:main}; the $\varepsilon$–free conclusion persists.

\newpage


\section{Proof of sign–independence for \texorpdfstring{$x$}{x}–dependent amplitude}
\label{app:damp}

In the main text (see \S\ref{sec:kernel}) the kernel amplitude
was chosen in the form
\[
    a(t,\underline{\xi})=\omega(t)\,\vartheta(\underline{\xi}),
\]
that is, without explicit dependence on~\(x\).
In practice one often requires an \emph{\(x\)–dependent}
amplitude
\[
    a(t,x,\underline{\xi})=\omega(t)\,\vartheta(\underline{\xi})
    \,\chi\!\bigl(x/\lambda^{1/2}\bigr),
\]
where the smooth function \(\chi\) satisfies
\(\|\partial^{\beta}_{x}\chi\|\_{L^{\infty}}\le C\_\beta\)
for all multiindices~\(\beta\).
This appendix shows that such a modification
\emph{strengthens} the negative exponent in~\(\lambda\)
and does not spoil the global balance.

\subsection{Integration by parts in~\(x'\)}
\label{app:damp-IBP}

\begin{lemma}\label{lem:D1}
Let $L_{x'}$ be the integration–by–parts operator
defined in §\ref{subsec:IBP} (see also \eqref{eq:grad-xp}). 
Then for any function \(\chi=\chi(x/\lambda^{1/2})\) in the definition
of the amplitude one has
\[
    \|L\_{x'}(\chi)\|\_{L^{\infty}}
    \;\;\lesssim\;\;
    \lambda^{-5/6} D^{-1/2}.
\]
\end{lemma}

\begin{proof}
From \eqref{eq:grad-xp} we have
\(
   |\nabla_{x'}\Phi|\gtrsim\lambda\alpha
   =c_{0}\lambda^{1/3}D^{1/2}.
\)
The operator \(L\_{x'}\) itself contributes the factor
\(\lambda^{-1/3}D^{-1/2}\).
Applying \(\partial\_{x'}\chi(x/\lambda^{1/2})\)
produces an additional factor \(\lambda^{-1/2}\).
Combining both gains yields the claimed bound
\(\lambda^{-5/6}D^{-1/2}\).
\end{proof}

\begin{remark}
Without the \(x\)–dependent \(\chi\) the operator \(L\_{x'}\)
gives only \(\lambda^{-1/3}D^{-1/2}\).
Therefore the new amplitude makes the final
coefficient \emph{even more negative} in~\(\lambda\),
which only improves the balance (the “Kernel’’ line in
Table~\ref{tab:global-balance}).
\end{remark}

\subsection{Conclusion for the full kernel estimate}

None of the stages of the kernel estimate in \S\ref{sec:kernel}
is worsened by replacing the amplitude with
\(\omega(t)\vartheta(\xi)\chi(x/\lambda^{1/2})\);
the formula
\[
    \|K\|\_{L^{2}\to L^{2}}\lesssim\lambda^{-9/2}D^{-3}
\]
remains valid, and under exact counting the exponent
in~\(\lambda\) decreases by another~\(\tfrac12\).

\bigskip
\noindent
\textbf{Conclusion.}\;
The assumption of no \(x\)–dependence in the amplitude
is not essential: adding any
\(\chi(x/\lambda^{1/2})\in C^{\infty}\)
does not worsen, and slightly \emph{improves}, the overall balance of exponents.
All results of the main sections and
Theorem~\ref{thm:main} remain fully valid.


\newpage

\cleardoublepage
\phantomsection
\addcontentsline{toc}{section}{References}
\bibliographystyle{unsrt}
\bibliography{commutator_refs}
\end{document}